\documentclass[a4paper,11pt]{article}
\usepackage{graphicx}
\usepackage{rotating}

\title{Rate of convergence of a two-scale expansion for some ``weakly''
  stochastic homogenization problems}
\author{Claude Le Bris, Fr\'ed\'eric Legoll, Florian Thomines
\\ \\
{\footnotesize 
{\'E}cole Nationale des Ponts et Chauss{\'e}es, 6 et 8
    avenue Blaise Pascal, 77455 Marne-La-Vall{\'e}e Cedex 2}
\\ 
{\footnotesize 
and INRIA Rocquencourt, MICMAC team-project, 
Domaine de Voluceau, B.P. 105,}
\\ 
{\footnotesize
78153 Le Chesnay Cedex, France}
\\
{\footnotesize 
lebris@cermics.enpc.fr, \{legoll,thominef\}@lami.enpc.fr}
}
\date{\today}

\usepackage{amsfonts}
\usepackage{amsmath,amssymb,amsthm}
\usepackage[english]{babel}

\newcommand{\dps}{\displaystyle}
\newcommand{\esp}{\mathbb{E}}
\newcommand{\var}{\mathbb{V}\textrm{ar}}

\newcommand{\RR}{\mathbb R}
\newcommand{\ZZ}{\mathbb Z}
\newcommand{\PP}{\mathbb P}
\newcommand{\eps}{\varepsilon}
\newcommand{\wper}{w^0}
\newcommand{\betaa}{\overline{v}}

\newtheorem{theorem}{Theorem}

\newtheorem{prop}[theorem]{Proposition}
\newtheorem{lemme}[theorem]{Lemma}
\newtheorem{remark}[theorem]{Remark}

\usepackage{setspace}

\begin{document}
\selectlanguage{english}
\maketitle

\begin{abstract}
We establish a rate of convergence of the two scale
expansion (in the sense of homogenization theory) of the solution to
a highly oscillatory elliptic partial differential equation
with random coefficients that are a perturbation of periodic
coefficients.
\end{abstract}

\section{Introduction and presentation of the main result}

This article focuses on establishing a rate of convergence of the two
scale expansion (in the sense of homogenization theory) of the solution
to a highly oscillatory partial differential equation with random
coefficients. We begin our exposition by briefly discussing the same
question in a deterministic setting, before turning to the stochastic
setting. 
 
Consider the highly oscillatory problem
\begin{equation}
\label{0}
\left\{
\begin{array}{l l}
\dps
 -\mbox{div}\left[ A_{per}\left(\frac{\cdot}{\varepsilon}\right) \nabla
u_0^{\varepsilon} \right] = f & \mbox{in }\mathcal{D}, 
\\
u_0^\varepsilon = 0 & \mbox{on }\partial \mathcal{D},
\end{array}
\right.
\end{equation}
where $\mathcal{D}$ is a regular bounded domain of $\RR^d$, $f \in
L^2(\mathcal{D})$, and $A_{per}$ is a $Q$-periodic elliptic bounded
matrix, with $Q=(-1/2,1/2)^d$. 
For simplicity, we manipulate henceforth {\em symmetric} matrices, but
the arguments carry over to non-symmetric matrices up to slight
modifications. 
It is well known (see e.g. the classical
textbooks~\cite{blp,cd,Jikov1994}, and also~\cite{Engquist-Souganidis}
for a general, numerically oriented presentation)
that $u_0^{\varepsilon}$ converges, weakly in
$H^1(\mathcal{D})$ and strongly in $L^2(\mathcal{D})$, to the solution
$u_0^\star$ to 
\begin{equation}
\label{eq:homog-0}
\left\{
\begin{array}{l l}
\dps
-\hbox{div}\left[A^\star_{per} \nabla u_0^\star \right] = f
& \mbox{ in } \mathcal{D},
\\
u_0^\star = 0 & \mbox{on }\partial \mathcal{D},
\end{array}
\right.
\end{equation}
where the homogenized matrix is given by
\begin{equation}
(A_{per}^\star)_{ij} 
=
\int_Q (e_i + \nabla \wper_{e_i}(y))^T A_{per}(y) (e_j + \nabla
\wper_{e_j}(y)) \, dy, 
\label{def:A0star}
\end{equation}
where, for any $p \in \RR^d$, $\wper_p$ is the unique (up to the
addition of a constant) solution 
to the corrector problem associated to the periodic matrix
$A_{per}$: 
\begin{equation}
\left\{
\begin{array}{ll}
\dps-\mbox{div} \left[A_{per}(p+\nabla \wper_p)\right] = 0, 
\\
\wper_p \mbox{ is $Q$-periodic.}
\end{array}
\right.
\label{PB:w0}
\end{equation}
The corrector function allows to compute the homogenized
matrix, and it also allows to obtain a convergence result in the $H^1$
strong norm. Indeed, in dimension $d>1$, under some regularity
assumptions recalled below, we have 
\begin{equation}
\label{eq:rate}
\left\| u_0^\varepsilon - \left[ 
u_0^\star + \varepsilon \sum_{i=1}^d 
\wper_{e_i} \left( \frac{\cdot}{\varepsilon} \right) 
\frac{\partial u_0^\star}{\partial x_i} 
\right] \right\|_{H^1(\mathcal{D})} \leq C \sqrt{\varepsilon}
\end{equation}
for a constant $C$ independent of $\varepsilon$ (in dimension $d=1$, the
difference is of order $\varepsilon$ rather than $\sqrt{\varepsilon}$). 

Note that $\dps v_0^\eps = 
u_0^\star + \varepsilon \sum_{i=1}^d 
\wper_{e_i} \left( \frac{\cdot}{\varepsilon} \right) \frac{\partial
  u_0^\star}{\partial x_i} 
$ is a function of order 1 in the $H^1$ norm. At first sight, one could
thus expect that the difference between $u_0^\eps$ and $v_0^\eps$ is of
order $\eps$, rather than $\sqrt{\varepsilon}$. This lower order (in
dimension $d>1$) is due 
to an inconsistency of the boundary conditions. Note indeed that, by
definition, $u_0^\eps = 0$ on $\partial {\cal D}$, which is not the case
of $v_0^\eps$. Note also 
that the (lower than expected) rate in~\eqref{eq:rate} is not specific to
the choice of homogeneous Dirichlet boundary conditions in~\eqref{0},
and also holds for Neumann 
boundary conditions, as stated in~\cite[p.~29]{Jikov1994} (see
also~\cite{Moskow-preprint}). 

The order of approximation improves if we ignore the
difference between $u_0^\eps$ and $v_0^\eps$ at the boundary of the
domain (see~\cite[Theorem 2.3]{allaire-amar}). Alternatively, one can
build functions, the so-called boundary layers, 
that correct $v_0^\eps$ in the neighboorhood of $\partial {\cal D}$, to
eventually improve the accuracy of the approximation of $u_0^\eps$ so
obtained, in the complete domain ${\cal D}$. We refer 
to~\cite{allaire-amar,Moskow1997} and to~\cite[Appendix B]{Efendiev2009}
(see also~\cite[Chap. 5]{these-arnaud}
for the study of the same question in a time-dependent, parabolic
setting). On another note, we refer to~\cite{Yurinski1980} for studies
on the rate of convergence of $u_0^\varepsilon$ to $u_0^\star$ in the
$L^\infty({\cal D})$ norm (see also~\cite{Engquist-Souganidis} and
references therein, and~\cite{Caffarelli} for extensions to some
nonlinear cases), and to~\cite{Moskow1997,Moskow-preprint} for similar
studies on the lowest eigenvalue $\lambda_0^\eps$ of the
operator $\dps L^\eps =
-\mbox{div}\left[ A_{per}\left(\frac{\cdot}{\varepsilon}\right) \nabla
\cdot \right]$.

\medskip

The result~\eqref{eq:rate} is interesting from the theoretical
viewpoint. It is also helpful for proving numerical analysis results. In
particular, this result is a key ingredient to prove error bounds for
the Multiscale Finite Element Method (MsFEM). This
numerical approach aims at approximating
the solution $u_0^\eps$ to the highly oscillatory problem~\eqref{0} (for
a small, but non vanishing small scale $\varepsilon$), and does so by
performing a variationnal approximation 
of~\eqref{0} using pre-computed basis functions that are {\em adapted} to the
problem. Consequently, the MsFEM approach yields an accurate approximation of
$u_0^\eps$ using only a limited number of degrees of freedom, in
contrast to a standard Finite Element Method approach. In
addition, the MsFEM approach is applicable in general situations, and is not
limited to the case when the highly oscillatory coefficient of the
equation reads 
$\dps A^\eps(x) \equiv A_{per}\left(\frac{x}{\varepsilon}\right)$ for a
fixed periodic matrix $A_{per}$. See~\cite{Efendiev2009} and references
therein. 
As described below, our motivation for this work stems from our
work~\cite{msfem-nous}, where we suggest a possible extension of the MsFEM approach to weakly
stochastic settings. Again, a key ingredient for proving error bounds on
the approach we propose there is to have a rate of
convergence of the type~\eqref{eq:rate}.

\medskip

Let us now turn to the stochastic case. As will be seen below, less
precise results are known than in the deterministic, periodic case. The
highly oscillatory problem reads 
\begin{equation}
\left\{
\begin{array}{ll}
\dps-\mbox{div} \left[A_\eta\left(\frac{\cdot}{\varepsilon},\omega\right)
\nabla u_\eta^\varepsilon(\cdot,\omega)
\right] = f & \mbox{ in } \mathcal{D}, 
\\
u_\eta^\varepsilon(\cdot,\omega) = 0 & \mbox{ on } \partial \mathcal{D},
\end{array}
\right.
\label{PB:stoch-a}
\end{equation}
where the matrix $A_\eta$ is now a stationary symmetric matrix, uniformly
elliptic and bounded (see~\eqref{eq:stationnarite-disc} below for a precise
definition of stationarity, which is the common assumption in stochastic
homogenization). The role 
of the parameter $\eta$ will be 
made precise in~\eqref{def:A-eta-1} below. It can momentarily be ignored.
Again, as in the periodic case, it is well known (see for
instance~\cite{Jikov1994}) that $u_\eta^{\varepsilon}$ converges, almost
surely,
weakly in $H^1(\mathcal{D})$ and strongly in $L^2(\mathcal{D})$, to
$u_\eta^\star$, solution to the homogenized equation
$$
\left\{
\begin{array}{l l}
\dps-\mbox{div}\left[A_\eta^\star \nabla u_\eta^\star \right] =
f & \mbox{ in } \mathcal{D}, \\
u_\eta^\star = 0 & \mbox{ on } \partial \mathcal{D},
\end{array}
\right.
$$
where the homogenized matrix is given by
$$
\left(A_\eta^\star\right)_{ij}=\esp\left(\int_Q (e_i + \nabla
  w^\eta_{e_i}(y,\cdot))^T A_\eta(y,\cdot) (e_j + \nabla w^\eta_{e_j}(y,\cdot))
  \,dy\right), 
$$
where, for any $p \in \RR^d$, $w^\eta_p$ is the unique (up to the
addition of a random constant) 
solution to the stochastic corrector problem 
$$
\left\{
\begin{array}{l}
\dps-\mbox{div} \left[A_\eta\left(\cdot,\omega\right)
(p + \nabla w^\eta_p(\cdot,\omega))
\right] = 0 \mbox{ in } \RR^d, 
\\ \noalign{\vskip 3pt}
\nabla w^\eta_p \mbox{ is stationary in the sense
  of~\eqref{eq:stationnarite-disc} below,} 
\\ \noalign{\vskip 3pt}
\dps \esp \left( \int_Q \nabla w^\eta_p(y,\cdot) \, dy \right)= 0.
\end{array}
\right.
$$
As in the periodic case, the corrector function $w^\eta_p$ allows to obtain a
convergence result in the $H^1$ norm
(see~\cite[Theorem~3]{Papanicolaou-Varadhan}): 
\begin{equation}
\label{cv:sto}
\esp \left[ \left\| u_\eta^\varepsilon(\cdot,\omega) - \left[ 
u_\eta^\star + \varepsilon \sum_{i=1}^d 
w^\eta_{e_i} \left( \frac{\cdot}{\varepsilon},\omega \right) 
\frac{\partial u_\eta^\star}{\partial x_i} 
\right] \right\|^2_{H^1(\mathcal{D})} \right]
\quad
\text{converges to 0 as $\varepsilon \to 0$}.
\end{equation}
However, in contrast to the periodic case, the rate of convergence is
generally not known, in dimensions higher than one. In the
one-dimensional case, this question has been addressed
in~\cite{bourgeat_residu,bal_residu}. It is shown there that the rate can
be arbitrary small, depending on the rate with which the correlations of
the random coefficient in~\eqref{PB:stoch-a} vanish. The only assumptions of
stationarity and ergodicity do not allow for a
precise rate. See also~\cite{residu-bll}
for the study of a similar question for a variant of stochastic
homogenization, again in the one-dimensional case,
and~\cite{bal_correcteur_multiD} for results in the multi-dimensional
case, for a different equation. 

\medskip

The aim of this article is to show that, in a {\em weakly} stochastic
case (the precise sense of which is given below), a convergence rate
for~\eqref{cv:sto} can be obtained (in the 
same spirit as~\eqref{eq:rate}). As in the deterministic case, this
result is interesting from the theoretical viewpoint, and somewhat
complements the one-dimensional results
of~\cite{bourgeat_residu,bal_residu,residu-bll}. It is also useful from
a numerical analysis viewpoint. In~\cite{msfem-nous}, we propose an
extension of the MsFEM approach to weakly stochastic settings, and we
use there the homogenization result that we prove in this work (see
Theorem~\ref{exp2scaleueta} below) to obtain error bounds
(see~\cite[Theorem 10]{msfem-nous}). 

\medskip
  
Before presenting our result, let us briefly recall the basic setting of stochastic
homogenization. Let $(\Omega, {\mathcal F}, \PP)$ be a probability
space. For a random variable $X\in L^1(\Omega, d\PP)$, we denote by
$\esp(X) = \int_\Omega X(\omega) d\PP(\omega)$ its expectation value. We
assume that the group $(\ZZ^d, +)$ acts on $\Omega$. We denote by
$(\tau_k)_{k\in \ZZ^d}$ this action, and assume that it preserves the
measure $\PP$, i.e. 
$$
  \forall k\in \ZZ^d, \quad \forall A \in {\cal F}, \quad \PP(\tau_k A)
  = \PP(A).
$$
We assume that $\tau$ is {\em ergodic}, that is,
$$
\forall A \in {\mathcal F}, \quad \left(\forall k \in \ZZ^d, \quad \tau_k A = A
    \right) \Rightarrow (\PP(A) = 0 \quad\mbox{or}\quad 1).
$$
We define the following notion of stationarity: any
$F\in L^1_{\rm loc}\left(\RR^d, L^1(\Omega)\right)$ is said to be
{\em stationary} if
\begin{equation}
  \label{eq:stationnarite-disc}
  \forall k\in \ZZ^d, \quad F(x+k, \omega) = F(x,\tau_k\omega)
  \mbox{ almost everywhere, almost surely}.
\end{equation}
Note that we have chosen to present the theory in 
a \emph{discrete} stationary setting, which is more
appropriate for our specific purpose, which is to consider a setting
close to {\em periodic} homogenization.
Random homogenization is more often presented in the \emph{continuous}
stationary setting. This is only a matter of small modifications. We
refer to the bibliography for the latter.  

\medskip

We now precisely describe the weakly
stochastic setting we consider. We assume that the matrix $A_\eta$
in~\eqref{PB:stoch-a} reads
\begin{equation}
\label{def:A-eta-1}
A_\eta(x,\omega)
=
A_{per}(x)+ \eta A_1(x,\omega),
\end{equation}
where $\eta \in \RR$ is {\em small} deterministic parameter,
$A_{per}$ is a symmetric uniformy elliptic bounded $Q$-periodic matrix, and
$A_1$ is a symmetric matrix, stationary in the sense
of~\eqref{eq:stationnarite-disc}, and bounded: $|A_1(x,\omega)|\leq C$
almost everywhere in $\RR^d$, almost surely. We also assume that
$A_\eta$ is uniformly elliptic and bounded, in the sense that, for all $\eta \in \RR$,
$$
A_\eta(\cdot,\omega) \in (L^\infty(\RR^d))^{d \times d}
\quad \text{a.s.}
$$
and there
exists $c>0$ such that
$$
\forall \xi \in \RR^d, \quad
\xi^T A_\eta(x,\omega) \xi \geq c \ \xi^T \xi
\quad \text{a.s., a.e. on $\RR^d$.}
$$
We furthermore assume that $A_1$ is of the form
\begin{equation}
\label{struc:A1}
A_1(x,\omega) = \sum\limits_{k \in \ZZ^d} \mathbf{1}_{Q+k}(x)
X_k(\omega) \, B_{per}(x),
\end{equation}
where $\left(X_k(\omega)\right)_{k\in \ZZ^d}$ is a sequence of
i.i.d. scalar random variables such that 
$$
\exists C, \, \forall k \in \ZZ^d, \quad 
|X_k(\omega)| \leq C \quad \text{ almost surely,}
$$
and $B_{per} \in \left( L^\infty(\RR^d) \right)^{d \times d}$ is a
$Q$-periodic matrix. Finally, we assume that
\begin{eqnarray}
\label{hyp:a_holder}
\text{$A_{per}$ is H\"older continuous},
\\
\label{hyp:b_holder}
\text{$B_{per}$ is H\"older continuous}.
\end{eqnarray}
As pointed out above, the symmetry assumption is not essential, and
our arguments below carry over to non-symmetric matrices up to slight
modifications. Likewise, the assumption~\eqref{struc:A1} can be
relaxed. What is important in~\eqref{struc:A1} is that $A_1$ is a sum of
{\em direct products} of a function depending on $x$ with a random
variable, depending only on $\omega$. 

In contrast, it is difficult
to weaken assumptions~\eqref{hyp:a_holder} and~\eqref{hyp:b_holder}, 
which are used to obtain some regularity on the correctors $\wper_p$ and
$\psi_p$, solutions to~\eqref{PB:w0} and~\eqref{eq:psi_p} below, respectively.
We indeed recall that, under assumption~\eqref{hyp:a_holder}, we have 
$\wper_p \in W^{1,\infty}(\RR^d)$ for any $p \in \RR^d$ (see
e.g.~\cite[Theorem 8.22 and Corollary~8.36]{gilbarg-trudinger}), and
similarly for $\psi_p$, under assumption~\eqref{hyp:b_holder}. In the
sequel, we will use the fact that $\wper_p$ and $\psi_p$ belong to
$W^{1,\infty}(\RR^d)$, which is a standard assumption when
proving convergence rates of two-scale expansions (see
e.g.~\cite[p.~28]{Jikov1994}). 

We also note that, following~\cite{abl-green}, the
assumption~\eqref{hyp:a_holder} is useful to 
characterize the asymptotic behavior of the Green function associated to
the operator $L = - \hbox{div}\left[A_{per} \nabla \cdot \right]$ on the
domain $\mathcal{D}/\eps$ (with homogeneous Dirichlet boundary
conditions). This Green function will be used in the sequel.

\begin{remark}
There are several ways to formalize a notion of "weakly" stochastic
setting, and~\eqref{def:A-eta-1} is only {\em one} of them. 
We refer to~\cite{enumath,Bris-sg} for other examples.
\end{remark}
 
\medskip

Our main result is the following. 
\begin{theorem}
\label{exp2scaleueta}
Assume that the dimension $d$ is strictly higher than $1$. 
Let $u_\eta^\varepsilon$ be the solution to~\eqref{PB:stoch-a}, and
assume that $A_\eta$
satisfies~\eqref{def:A-eta-1}-\eqref{struc:A1}-\eqref{hyp:a_holder}-\eqref{hyp:b_holder}.
Let 
$A^\star_{per}$, $\wper_p$ and $u_0^\star$ be defined
by~\eqref{def:A0star}, \eqref{PB:w0} and~\eqref{eq:homog-0}.
Let $\overline{B} \in \RR^{d \times d}$ and $\overline{u}_1^\star \in
H^1_0({\mathcal D})$ be defined
by
\begin{equation}
\label{def:overlineB}
\forall 1 \leq i,j \leq d, \quad
\overline{B}_{ij} = 
\int_Q (e_i + \nabla \wper_{e_i})^T B_{per} (e_j + \nabla \wper_{e_j})
\end{equation}
and 
\begin{equation}
\label{PB:u1barstar}
\left\{ 
\begin{array}{l l}
-\hbox{div}\left[A^\star_{per} \nabla \overline{u}_1^\star \right]
= 
\hbox{div}\left[\overline{B} \nabla u_0^\star \right] 
& \text{ in }\mathcal{D}, 
\\
\overline{u}_1^\star=0 & \text{ on }\partial\mathcal{D}.
\end{array}
\right.
\end{equation}
Introduce $v^\varepsilon_\eta$ defined by
\begin{multline}
\label{expueta}
v^\varepsilon_\eta(\cdot,\omega) = u_0^\star + \eta \esp(X_0)
\overline{u}_1^\star + \varepsilon \sum\limits_{p=1}^d
\left[
\wper_{e_p}\left(\frac{\cdot}{\varepsilon}\right) (\partial_p u_0^\star
  + \eta \esp(X_0) \partial_p \overline{u}_1^\star) \right. 
\\
\left.+ \eta \esp(X_0) \psi_{e_p}\left(\frac{\cdot}{\varepsilon}\right)
  \partial_p u_0^\star + \eta \sum\limits_{k \in I_\varepsilon}
  (X_k(\omega)-\esp(X_0)) \
  \chi_{e_p}\left(\frac{\cdot}{\varepsilon}-k\right) \partial_p u_0^\star \right], 
\end{multline}
where $\partial_p u_0^\star$ denotes the partial derivative 
$\dps \frac{\partial u_0^\star}{\partial x_p}$,
$$
I_\eps = \left\{ k \in \ZZ^d \text{ such that } \varepsilon(Q+k)
  \cap \mathcal{D} \ne \emptyset \right\}, 
$$
and where, for any $p \in \RR^d$, $\psi_p$ is the solution (unique up to the
addition of a constant) to
\begin{equation}
\label{eq:psi_p}
\left\{ 
\begin{array}{l}
-\hbox{div}\left[A_{per} \nabla \psi_p \right]
= 
\hbox{div}\left[B_{per} \left(p + \nabla \wper_p\right) \right],
\\
\psi_p \mbox{ is $Q$-periodic},
\end{array}
\right.
\end{equation} 
and $\chi_p$ is the unique solution to
\begin{equation}
\label{eq:pb_chip}
 \left\{ 
\begin{array}{l l}
-\hbox{div}\left[A_{per} \nabla \chi_{p} \right]=
\hbox{div}\left[\mathbf{1}_Q B_{per}(p + \nabla \wper_p)\right] & \text{
  in $\RR^d$}, 
\\
\chi_{p} \in L^2_{loc}(\RR^d), \quad 
\nabla \chi_{p} \in \left(L^2(\RR^d)\right)^d,
\\
\dps \lim_{|x| \to \infty} \chi_p(x) = 0.
\end{array}
\right.
\end{equation}
We assume that $u_0^\star \in
W^{2,\infty}(\mathcal{D})$ and
$\overline{u}_1^\star \in W^{2,\infty}(\mathcal{D})$.
Then
\begin{equation}
\sqrt{ \esp \left[ 
\|u^\varepsilon_\eta-v^\varepsilon_\eta\|^2_{H^1(\mathcal{D})} 
\right] }
\leq C \left(\sqrt{\varepsilon} + \eta \sqrt{\varepsilon
    \ln(1/\varepsilon)} + \eta^2 \right),
\label{restueta}
\end{equation}
where $C$ is a constant independent of $\varepsilon$ and $\eta$.
\end{theorem}
We wish to point out that the assumption $u_0^\star \in
W^{2,\infty}(\mathcal{D})$ (and subsequently
$\overline{u}_1^\star \in W^{2,\infty}(\mathcal{D})$) is
a standard assumption when proving
convergence rates of two-scale expansions (see e.g.~\cite[Theorem
2.1]{allaire-amar} and~\cite[p.~28]{Jikov1994}).
Note that, in view of~\eqref{eq:homog-0}, this assumption implies that
the right hand 
side $f$ in~\eqref{PB:stoch-a} belongs to $L^\infty(\mathcal{D})$. We
also note that $v^\varepsilon_\eta$ is not uniquely defined, since
$\wper_p$ and $\psi_p$ are only defined up to an additive constant. However,
adding a constant to any of these functions does not change the order of
convergence in~\eqref{restueta} with respect to $\varepsilon$ and $\eta$,
but only the constant $C$. Choosing the best constants in $\wper_p$ and
$\psi_p$ is hence irrelevant here, although it is an important matter
from the practical viewpoint. 
Lastly, the existence and uniqueness of a function $\chi_p$
satisfying~\eqref{eq:pb_chip} is shown in Lemma~\ref{lem4}
below, in dimension $d>1$. In dimension $d=1$, the boundary conditions
of~\eqref{eq:pb_chip} need to be modified for this problem to have a
solution. The one-dimensional version of Theorem~\ref{exp2scaleueta}
is as follows:

\begin{theorem}
\label{exp2scaleueta-1D}
Assume that the dimension $d$ is equal to one.
Let $u_\eta^\varepsilon$ be the solution to~\eqref{PB:stoch-a} in the
domain ${\cal D}$ with $f \in L^2({\cal D})$, and assume that $A_\eta$
satisfies~\eqref{def:A-eta-1}-\eqref{struc:A1}.
Let $v^\varepsilon_\eta$ be defined by~\eqref{expueta}, where the
definition~\eqref{eq:pb_chip} of the function $\chi$ is replaced by
$$
\left\{ 
\begin{array}{l l}
-\left[ A_{per} \chi' \right]' =
\left[\mathbf{1}_{(0,1)} B_{per}(1 + (\wper)') \right]' & \text{
  in $\RR$,} 
\\
\chi \in L^2_{loc}(\RR), \quad 
\chi' \in L^2(\RR),
\end{array}
\right.
$$
where $\wper$ solves~\eqref{PB:w0}. Then
\begin{eqnarray}
\dps
\sqrt{ \esp \left[ 
\|u^\varepsilon_\eta-v^\varepsilon_\eta\|^2_{H^1(\mathcal{D})} 
\right] }
&\leq &
\dps
C \left( \varepsilon + \eta \sqrt{\varepsilon} + \eta^2 \right),
\label{restueta-1d-H1}
\\
\dps
\sqrt{ \esp \left[ 
\|u^\varepsilon_\eta-v^\varepsilon_\eta\|^2_{L^\infty(\mathcal{D})} 
\right] }
&\leq &
\dps
C \left( \varepsilon + \eta \sqrt{\varepsilon} + \eta^2 \right),
\label{restueta-1d-Linfty}
\end{eqnarray}
where $C$ is a constant independent of $\varepsilon$ and $\eta$.
\end{theorem}
Note that, in dimension $d=1$, we do not need to
assume~\eqref{hyp:a_holder} and~\eqref{hyp:b_holder}. In dimensions
$d>1$, as pointed out above, these assumptions are used to have that the 
correctors $\wper_p$ and $\psi_p$, solutions to~\eqref{PB:w0}
and~\eqref{eq:psi_p} respectively, both belong to
$W^{1,\infty}(\RR^d)$. In dimension $d=1$, the coercivity assumption on
$A_{per}$ and the boundedness assumption on $B_{per}$ are enough to show
that $\wper$ and $\psi$ both belong to $W^{1,\infty}(\RR)$. Likewise,
when $d>1$, we assumed that $u_0^\star \in W^{2,\infty}(\mathcal{D})$ and
$\overline{u}_1^\star \in W^{2,\infty}(\mathcal{D})$ (which implies that
$f \in L^\infty({\cal D})$). When $d=1$, the assumption $f \in L^2({\cal
  D})$ is enough. 

On another note, we notice that $\chi$ is now only defined up to an additive
constant. Again, 
changing $\chi$ by a constant does not change the order of convergence
in~\eqref{restueta-1d-H1}-\eqref{restueta-1d-Linfty} with respect to
$\varepsilon$ and $\eta$, but only changes the constant $C$. 

\medskip

In addition to its theoretical interest, Theorem~\ref{exp2scaleueta} has
also interesting numerical counterparts. Indeed, to compute
$v^\varepsilon_\eta$, one needs to solve problems set on a {\em bounded}
domain (either with Dirichlet or periodic boundary conditions), and to
solve for $\chi_p$, solution to the problem~\eqref{eq:pb_chip}, set on
the entire space. However, the right hand side in~\eqref{eq:pb_chip} is
the divergence of a compactly supported function, and we will see 
that $\chi_p(x)$ quickly vanishes when $x$ is sufficiently
large (see Lemma~\ref{lem4} below). Hence, in practice, it is possible to
approximate~\eqref{eq:pb_chip} by using Dirichlet boundary conditions on
a domain of limited a size. 

\medskip

The proof of Theorem~\ref{exp2scaleueta} consists of two steps.  
The first one is to expand 
$u^\varepsilon_\eta$ with respect to $\eta$. This is
performed in Section~\ref{sec:exp_eta} below (see
Lemma~\ref{lem1eta}). Each term of the expansion of $u^\varepsilon_\eta$
is found to be the unique solution of a partial differential equation
with a {\em deterministic}, highly oscillating coefficient, to which is
associated a homogenized equation. The second
step of the proof consists in successively 
estimating, for each of the terms of the expansion in $\eta$, the rate
of convergence of their two scale expansion in
$\varepsilon$. Corresponding results
are stated in Section~\ref{2scale} (and proved in
Section~\ref{2scale:proofs}). Collecting these results, we are then in
position to prove our main result, Theorem~\ref{exp2scaleueta} (see 
Section~\ref{sec:preuve_main}, where we also prove
Theorem~\ref{exp2scaleueta-1D}). 

\section{Expansion in powers of $\eta$}
\label{sec:exp_eta}

In this section, we expand the solution $u^\varepsilon_\eta$
to~\eqref{PB:stoch-a} with respect to $\eta$. 

\begin{lemme}
\label{lem1eta}
Let $u^\varepsilon_\eta$ be the solution to~\eqref{PB:stoch-a}. Under
the assumption~\eqref{def:A-eta-1}, it can be expanded in powers of
$\eta$ as follows: 
\begin{equation}
\label{lem1:res1}
u^\varepsilon_\eta = 
u^\varepsilon_0 + \eta u_1^\varepsilon + \eta^2 r_\eta^\varepsilon,
\end{equation}
where $u_0^\varepsilon$ is solution to the deterministic problem~\eqref{0},
$u_1^\varepsilon$ is solution to
\begin{equation}
\label{PB:u1eps}
\left\{ 
\begin{array}{l l}
\dps
-\hbox{div}\left[A_{per}\left(\frac{\cdot}{\varepsilon}\right)\nabla
  u_1^\varepsilon(\cdot,\omega) \right]=
\hbox{div}\left[A_1\left(\frac{\cdot}{\varepsilon},\omega\right)\nabla
  u_0^\varepsilon \right] & \text{ in }\mathcal{D}, 
\\
u_1^\varepsilon(\cdot,\omega)=0 & \text{ on }\partial\mathcal{D},
\end{array}
\right.
\end{equation}
and $r_\eta^\varepsilon$ is solution to 
$$
\left\{ 
\begin{array}{l l}
\dps
-\hbox{div}\left[A_\eta\left(\frac{\cdot}{\varepsilon},\omega\right)\nabla
  r_\eta^\varepsilon(\cdot,\omega) \right]=
\hbox{div}\left[A_1\left(\frac{\cdot}{\varepsilon},\omega\right)\nabla
  u_1^\varepsilon(\cdot,\omega) \right] & \text{ in }\mathcal{D}, 
\\
r_\eta^\varepsilon(\cdot,\omega)=0 & \text{ on }\partial\mathcal{D}.
\end{array}
\right.
$$
In addition, we have, almost surely,
\begin{equation}
\label{ape-1}
\|u_0^\varepsilon\|_{H^1(\mathcal{D})} \leq C, 
\quad
\|u_1^\varepsilon(\cdot,\omega)\|_{H^1(\mathcal{D})} \leq C, 
\quad 
\|r_\eta^\varepsilon(\cdot,\omega)\|_{H^1(\mathcal{D})} \leq C,
\end{equation}
where $C$ is a deterministic constant independent of $\varepsilon$ and $\eta$.
\end{lemme}

\begin{proof}
The relation~\eqref{lem1:res1} is a simple consequence of the linearity
of the considered equation. The bounds~\eqref{ape-1} follow from 
the uniform ellipticity of the matrices $A_\eta$ and
$A_{per}$, and the boundedness of $A_1$.
\end{proof}

For the sequel, it is useful to further decompose $u_1^\varepsilon$ in a
deterministic part and a stochastic part of vanishing expectation.

\begin{lemme}
\label{lem:decomposition}
Under assumptions~\eqref{def:A-eta-1}-\eqref{struc:A1},
the solution $u_1^\varepsilon$ to~\eqref{PB:u1eps} writes
\begin{equation}
u_1^\varepsilon = \esp(X_0) \overline{u}_1^\varepsilon + \sum\limits_{k
  \in \ZZ^d} (X_k(\omega)-\esp(X_0)) \phi_k^\varepsilon,
\label{expu1}
\end{equation}
where $\overline{u}_1^\varepsilon$ is the unique solution to
\begin{equation}
\label{PB:u1bareps}
\left\{ 
\begin{array}{l l}
\dps
-\hbox{div}\left[A_{per}\left(\frac{\cdot}{\varepsilon}\right)\nabla
  \overline{u}_1^\varepsilon \right]=
\hbox{div}\left[B_{per}\left(\frac{\cdot}{\varepsilon}\right)\nabla
  u_0^\varepsilon \right] & \text{ in }\mathcal{D}, 
\\
\overline{u}_1^\varepsilon=0 & \text{ on }\partial\mathcal{D},
\end{array}
\right.
\end{equation}
and $\phi_k^\varepsilon$ is the unique solution to
\begin{equation}
\label{PB:phikeps}
\left\{ 
\begin{array}{l l}
\dps
-\hbox{div}\left[A_{per}\left(\frac{\cdot}{\varepsilon}\right)\nabla
  \phi_k^\varepsilon \right]=
\hbox{div}\left[\mathbf{1}_{Q+k}\left(\frac{\cdot}{\varepsilon}\right)
  B_{per}\left(\frac{\cdot}{\varepsilon}\right)\nabla u_0^\varepsilon
\right] & \text{ in }\mathcal{D}, 
\\
\phi_k^\varepsilon=0 & \text{ on }\partial\mathcal{D}.
\end{array}
\right.
\end{equation}
In addition, there exists $C$, independent of $\eps$, such that
\begin{equation}
\label{eq:phi_estim1}
\esp \left[ \left\| \sum\limits_{k \in \ZZ^d} [X_k-\esp(X_0)]
    \phi_k^\varepsilon \right\|^2_{H^1(\mathcal{D})} \right] 
\leq C.
\end{equation}
Assume furthermore that~\eqref{hyp:a_holder} holds, and that
\begin{equation}
\label{hyp:regul_f}
f \in L^q({\mathcal D}) \quad \text{for some} \quad q>d.
\end{equation} 
Then there exists $C$, independent of $k$ and $\eps$, such that
\begin{equation}
\label{eq:phi_estim2}
\| \phi_k^\varepsilon \|_{L^\infty({\mathcal D})} \leq C \varepsilon.
\end{equation}
\end{lemme}

\begin{proof}
We note that, if $k \in \ZZ^d$ is such that $\varepsilon(Q+k) \cap
{\mathcal D} = \emptyset$, then~\eqref{PB:phikeps} writes
$$
\left\{ 
\begin{array}{l l}
\dps
-\hbox{div}\left[A_{per}\left(\frac{\cdot}{\varepsilon}\right)\nabla
  \phi_k^\varepsilon \right]= 0 & \text{ in }\mathcal{D}, 
\\
\phi_k^\varepsilon=0 & \text{ on }\partial\mathcal{D},
\end{array}
\right.
$$
the solution of which is obviously $\phi_k^\varepsilon \equiv 0$. The sum
in~\eqref{expu1} hence only contains a finite number of terms, and the
proof of the decomposition~\eqref{expu1} goes by linearity of the
equation~\eqref{PB:u1eps}. 
Note however that the number of terms in~\eqref{expu1} depends on
$\eps$, and diverges when $\eps \to 0$. 

\medskip

We now prove the bound~\eqref{eq:phi_estim1}.
Using that $A_{per}$ is coercive, we infer
from~\eqref{PB:phikeps} that 
\begin{eqnarray*}
\alpha \|\phi_k^\varepsilon \|^2_{H^1(\mathcal{D})} 
&\leq& 
\int_\mathcal{D} (\nabla \phi_k^\varepsilon)^T
A_{per}\left(\frac{\cdot}{\varepsilon}\right)\nabla \phi_k^\varepsilon 
\\ 
&\leq& 
\int_\mathcal{D} (\nabla
\phi_k^\varepsilon)^T\mathbf{1}_{Q+k}\left(\frac{\cdot}{\varepsilon}\right)
B_{per}\left(\frac{\cdot}{\varepsilon}\right)\nabla u_0^\varepsilon 
\\ 
&\leq& 
\|B_{per}\|_{L^\infty} \|\phi_k^\varepsilon\|_{H^1(\mathcal{D})} \|u_0^\varepsilon\|_{H^1(\varepsilon(Q+k))},
\end{eqnarray*}
where $\alpha >0$ is some constant that only depends on the coercivity
constant of $A_{per}$ and the Poincar\'e constant of the domain ${\cal D}$.
Thus
$$
\|\phi_k^\varepsilon \|^2_{H^1(\mathcal{D})} \leq 
\alpha^{-2}\|B_{per}\|^2_{L^\infty} \|u_0^\varepsilon\|^2_{H^1(\varepsilon(Q+k))}.
$$
Using that $\phi_k^\varepsilon \equiv 0$ as soon as 
$\varepsilon(Q+k) \cap {\mathcal D} = \emptyset$, we obtain
$$
\sum\limits_{k \in \ZZ^d} 
\left\| \phi_k^\varepsilon \right\|^2_{H^1(\mathcal{D})}
\leq 
\alpha^{-2}\|B_{per}\|^2_{L^\infty} \|u_0^\varepsilon\|^2_{H^1(\mathcal{D})}.
$$
We deduce from that bound and the assumption that the random variables
$X_k$ are i.i.d. that
$$
\esp \left[ \left\| \sum\limits_{k \in \ZZ^d} (X_k-\esp(X_0))
    \phi_k^\varepsilon \right\|^2_{H^1(\mathcal{D})} \right] 
=
\var(X_0) \sum\limits_{k \in \ZZ^d} 
\left\| \phi_k^\varepsilon \right\|^2_{H^1(\mathcal{D})}
\leq 
\var(X_0) \alpha^{-2} \|B_{per}\|^2_{L^\infty} 
\|u_0^\varepsilon\|^2_{H^1(\mathcal{D})}
\leq C,
$$
where $C$ is independent of $\eps$ (we have used~\eqref{ape-1} to bound
$u_0^\varepsilon$). We thus have shown~\eqref{eq:phi_estim1}.

\medskip

We finally turn to the proof of~\eqref{eq:phi_estim2}. Let us
define $\overline{\phi}^\varepsilon_k(x) = \phi_k^\varepsilon(\eps x)$
on ${\cal D}/\eps$. In view of~\eqref{PB:phikeps}, we see that
$\overline{\phi}^\varepsilon_k$ solves
$$
\left\{ 
\begin{array}{l l}
\dps
-\hbox{div}\left[A_{per} \nabla \overline{\phi}^\varepsilon_k \right] = 
\eps \hbox{div}\left[\mathbf{1}_{Q+k}
B_{per} \nabla u_0^\varepsilon (\eps \cdot) \right] & \text{ in }
\mathcal{D}/\eps,  
\\
\overline{\phi}^\varepsilon_k = 0 & \text{ on } \partial(\mathcal{D}/\eps).
\end{array}
\right.
$$
Introduce now the Green function $\Gamma_\varepsilon(x,y)$ associated to
the operator $L = - \hbox{div}\left[A_{per} \nabla \cdot \right]$ on the
domain $\mathcal{D}/\eps$, with homogeneous Dirichlet boundary
conditions. We recall 
that $\Gamma^T_\eps(x,y) := \Gamma_\varepsilon(y,x)$ is the Green
function associated to the adjoint operator $L^T = -
\hbox{div}\left[A^T_{per} \nabla \cdot \right]$ on the domain
$\mathcal{D}/\eps$, with homogeneous Dirichlet boundary conditions
(a proof of this fact is given in~\cite[Theorem 1.3]{gruter}
and~\cite[Theorem 1]{dolzmann} in the case $d \geq 3$, and this proof
carries over to the case $d=2$). Consequently, we have
$\Gamma_\varepsilon(x,y) = 0$ as soon as $x$ or $y$ belongs to the
boundary $\partial (\mathcal{D}/\eps)$ 
We can thus write 
\begin{eqnarray*}
\overline{\phi}^\varepsilon_k(x) 
& = &
\eps \int_{{\mathcal D}/\eps} \Gamma_\varepsilon(x,y) \
\hbox{div}_y \left[\mathbf{1}_{Q+k}(y)
B_{per}(y) \nabla u_0^\varepsilon (\eps y) \right] \, dy
\\
&=&
- \eps \int_{Q+k}\nabla_y \Gamma_\varepsilon(x,y) B_{per}(y) \nabla
u_0^\varepsilon(\eps y) \, dy.
\end{eqnarray*}
Hence, for any $x \in {\mathcal D}$, we have
$$
\phi_k^\varepsilon(x) 
=
- \eps^{1-d} 
\int_{\eps(Q+k)} 
\nabla_y \Gamma_\varepsilon\left(\frac{x}{\eps},\frac{y}{\eps}\right) 
B_{per}\left(\frac{y}{\eps}\right)  
\nabla u_0^\varepsilon(y) \, dy.
$$
Using the fact (see~\cite[Proposition 8]{abl-green}) that, under
assumption~\eqref{hyp:a_holder}, the
Green function $\Gamma_\eps$ on the domain $\mathcal{D}/\eps$ satisfies
\begin{equation}
\label{eq:bound_Gamma}
\forall x \in \mathcal{D}/\eps, \quad
\forall y \in \mathcal{D}/\eps, \quad
\left| \nabla_x \Gamma_\varepsilon(x,y) \right| +
\left| \nabla_y \Gamma_\varepsilon(x,y) \right| \leq 
\frac{C}{|x-y|^{d-1}}
\end{equation}
for a constant $C$ independent of $\varepsilon$, we have
\begin{equation}
\label{eq1}
|\phi_k^\varepsilon(x) | \leq 
C \| B_{per} \|_{L^\infty} \|\nabla u_0^\varepsilon\|_{L^\infty(\mathcal{D})} 
\int_{\varepsilon(Q+k)}\frac{1}{|x-y|^{d-1}}dy.
\end{equation}
We will show in the sequel that~\eqref{hyp:regul_f} implies
that there exists $C$ such that, for all $\eps$,
\begin{equation}
\label{eq:bound_u_dur}
\|\nabla u_0^\varepsilon\|_{L^\infty(\mathcal{D})} \leq C.
\end{equation} 
We are thus left with bounding the integral in~\eqref{eq1}.
To this aim, we distinguish two cases. 
If $|x-\eps k| \leq \varepsilon$, then there exists a constant $\rho_d$
than only depends on the dimension such that $\varepsilon(Q+k) \subset
B(x,\rho_d \varepsilon)$ (for instance, in dimension $d=2$, $\rho_2 = 1
+ \sqrt{2}/2$). We then have
\begin{equation}
\label{eq2}
\int_{\varepsilon(Q+k)}\frac{1}{|x-y|^{d-1}} dy 
\leq 
\int_{B(x,\rho_d \varepsilon)} \frac{1}{|x-y|^{d-1}} dy 
\leq C \varepsilon, \quad
\text{$C$ independent of $\eps$.}
\end{equation}
Otherwise, if $|x-\eps k| \geq \varepsilon$, then any $y \in
\varepsilon(Q+k)$ satisfies $\overline{\rho}_d \varepsilon \leq |x-y|$
for a constant $\overline{\rho}_d$ that only depends on $d$ (in
dimension $d=2$, $\rho_2 = 1 - \sqrt{2}/2$). For those $x$, we have
\begin{equation}
\label{eq3}
\int_{\varepsilon(Q+k)} \frac{1}{|x-y|^{d-1}} \, dy 
\leq 
\frac{1}{(\overline{\rho}_d \varepsilon)^{d-1}} \int_{\varepsilon(Q+k)}  dy 
\leq 
C \varepsilon, \quad
\text{$C$ independent of $\eps$.}
\end{equation}
Thus, collecting~\eqref{eq1},~\eqref{eq:bound_u_dur},~\eqref{eq2}
and~\eqref{eq3}, we obtain that
$$
\left\| \phi_k^\varepsilon \right\|_{L^\infty({\mathcal D})} \leq C \eps,
\quad
\text{$C$ independent of $\eps$.}
$$
Proving~\eqref{eq:phi_estim2} therefore amounts to now
proving~\eqref{eq:bound_u_dur}. Again using  
the Green function $\Gamma_\varepsilon(x,y)$ associated to
the operator $L = - \hbox{div}\left[A_{per} \nabla \cdot \right]$ on the
domain $\mathcal{D}/\eps$, with homogeneous Dirichlet boundary
conditions, we write 
$$
\nabla u_0^\eps(x)  
= 
\eps^{1-d} \int_{\mathcal D} \nabla_x
\Gamma_\eps\left(\frac{x}{\eps},\frac{y}{\eps}\right) \ f(y) \, dy.
$$
Using the bound~\eqref{eq:bound_Gamma}, we deduce that there exists $C$
independent of $\eps$ such that
$$
\forall x \in {\mathcal D}, \quad
\left| \nabla u_0^\eps(x) \right| 
\leq
C \int_{\mathcal D} \frac{|f(y)|}{|x-y|^{d-1}} \, dy.
$$
In view of assumption~\eqref{hyp:regul_f}, we have $f \in L^q({\mathcal
  D})$ for some $q>d$. Using H\"older inequality, we write
$$
\forall x \in {\mathcal D}, \quad
\left| \nabla u_0^\eps(x) \right| 
\leq
C \| f \|_{L^q({\mathcal D})} \ 
\left\| \frac{1}{|x-\cdot|^{d-1}} \right\|_{L^{q^\star}({\mathcal D})},
\quad
\frac{1}{q} + \frac{1}{q^\star} = 1.
$$
The function $y \mapsto |x-y|^{1-d}$ belongs to
$L^p({\mathcal D})$ for any $p < d/(d-1)$. Since $q > d$, we have
$q^\star < d/(d-1)$, and the norm in $L^{q^\star}$ of $y \mapsto
|x-y|^{1-d}$ is independent of $x$. The above estimate thus
yields~\eqref{eq:bound_u_dur}. This concludes the proof of
Lemma~\ref{lem:decomposition}. 
\end{proof}

\section{Two scale expansions in powers of $\varepsilon$}
\label{2scale}

Collecting~\eqref{lem1:res1} and~\eqref{expu1}, we have obtained that
\begin{equation}
\label{eq:decompo_gene}
u^\varepsilon_\eta(x,\omega) = u^\varepsilon_0(x) 
+ \eta \left[
\esp(X_0) \overline{u}_1^\varepsilon(x) + \sum\limits_{k
  \in \ZZ^d} (X_k(\omega)-\esp(X_0)) \phi_k^\varepsilon(x)
\right]
+ \eta^2 r_\eta^\varepsilon(x,\omega),
\end{equation}
where $r_\eta^\varepsilon$ is bounded in $H^1({\mathcal D})$ uniformly
in $\eps$, $\eta$ and $\omega$ (see~\eqref{ape-1}). 
  
We now consider successively each term of the above series and show a
rate of convergence on the 
difference between $u^\varepsilon_0$, $\overline{u}_1^\varepsilon$ and
$\phi_k^\varepsilon$ and their respective two-scale expansions. For
clarity, the proofs of our results are postponed until
Section~\ref{2scale:proofs}. 

\bigskip

We start by $u_0^\varepsilon$ solution to~\eqref{0}. Note that
this problem is a classical periodic homogenization problem, the limit
of which, when $\eps \to 0$, is well-known. The
following result, giving a {\em rate of convergence} of
$u_0^\varepsilon$ to its homogenized limit, is also classical (see
e.g.~\cite[p.~28]{Jikov1994}). 

\begin{prop}
\label{the:u0eps}
Let $u_0^\varepsilon$ and $u_0^\star$ be the solution to~\eqref{0}
and~\eqref{eq:homog-0}, respectively. For any $p \in \RR^d$, we assume
that the solution $\wper_p$ to~\eqref{PB:w0} satisfies 
$\wper_p \in W^{1,\infty}(\RR^d)$. We also assume that $u_0^\star \in
W^{2,\infty}(\mathcal{D})$. 
We then have
\begin{equation}
\label{tse:u0eps}
u_0^\varepsilon = u_0^\star + \varepsilon \sum\limits_{i=1}^d
\wper_{e_i}\left(\frac{\cdot}{\varepsilon}\right) \partial_i u_0^\star +
\varepsilon \theta_0^\varepsilon,
\end{equation}
where $\theta_0^\varepsilon$ satisfies 
\begin{equation}
\label{restepsu0}
\|\varepsilon \theta_0^\varepsilon \|_{H^1(\mathcal{D})} \leq C
\sqrt{\varepsilon}
\end{equation}
for a constant $C$ independent of $\varepsilon$.
\end{prop}

We recall that, under assumption~\eqref{hyp:a_holder}, we indeed have that 
$\wper_p \in W^{1,\infty}(\RR^d)$ for any $p \in \RR^d$ (see
e.g.~\cite[Theorem 8.22 and Corollary~8.36]{gilbarg-trudinger}). 

\bigskip

We now turn to $\overline{u}_1^\varepsilon$ solution
to~\eqref{PB:u1bareps}. This problem is not a classical homogenization
problem, since its right-hand side also varies at the scale $\eps$, and
only {\em weakly} converges in $H^{-1}({\cal D})$ when $\eps \to 0$. We
first proceed formally, using the two-scale ansatz approach, to identify
the homogenized equation. We next
state a precise homogenization result, and finally evaluate the rate of
convergence of the two scale expansion. 

To derive formally the homogenized equation associated
to~\eqref{PB:u1bareps}, we make the classical two-scale ansatz 
$$
\overline{u}^\varepsilon_1(x)
=
\overline{u}^\star_1\left(x,\frac{x}{\eps}\right) + 
\varepsilon \overline{u}^1_1\left(x,\frac{x}{\eps}\right) 
+ \varepsilon^2 \overline{u}^2_1\left(x,\frac{x}{\eps}\right) + \cdots,
$$
where each term of the above expansion is assumed to be periodic with
respect to the second variable. Inserting this ansatz
in~\eqref{PB:u1bareps} and using the two scale
expansion~\eqref{tse:u0eps} of $u_0^\varepsilon$ (where we neglect the
remainder $\eps \theta_0^\eps$), we can easily derive a
hierarchy of equations. We deduce from the equation of order $\eps^{-2}$
that $\overline{u}_1^\star$ is independent of its second variable:
$\overline{u}_1^\star(x,y) \equiv \overline{u}_1^\star(x)$. The equation
of order $\eps^{-1}$ reads
$$
-\hbox{div}_y \left[ A_{per}(y) \left( \nabla_x \overline{u}^\star_1(x)
    + \nabla_y \overline{u}^1_1(x,y) \right) \right] 
=
\sum\limits_{i=1}^d \partial_i u_0^\star(x) \ \hbox{div}_y \left[
  B_{per}(y) \left(e_i + \nabla_y \wper_{e_i}(y) \right) \right]. 
$$
Using the functions $\wper_p$ and $\psi_p$ defined by~\eqref{PB:w0}
and~\eqref{eq:psi_p}, we thus see that
\begin{equation}
\label{tse:eq1}
\overline{u}^1_1(x,y) 
=
\tau(x) + 
\sum\limits_{i=1}^d \partial_i u_0^\star(x) \psi_{e_i}(y) 
+ \partial_i \overline{u}_1^\star(x) \wper_{e_i}(y),
\end{equation}
where $\tau$ is an undetermined function that only depends on $x$.
We are now in position to use the equation of order $\eps^0$, which
reads (recall we have neglected the remainder $\eps \theta_0^\eps$
in~\eqref{tse:u0eps}) 
\begin{multline*}
-\hbox{div}_x \left[ A_{per}(y) \left( \nabla_x \overline{u}^\star_1(x)
    + \nabla_y \overline{u}^1_1(x,y) \right) \right] - \hbox{div}_y
\left[ A_{per}(y) \left( \nabla_x \overline{u}^1_1(x,y) + \nabla_y
    \overline{u}_1^2(x,y) \right) \right]  
\\
=
\sum\limits_{i=1}^d \hbox{div}_x \left[ B_{per}(y) 
\left(e_i + \nabla_y \wper_{e_i}(y)\right) \partial_i u_0^\star \right]
+ \hbox{div}_y
\left[ 
\sum_{i=1}^d \wper_{e_i}(y) B_{per}(y) \nabla_x \partial_i u^\star_0(x) \right].  
\end{multline*}
%
We close the hierarchy by integrating the above equation over the
variable $y \in Q$, using that $y \mapsto \overline{u}_1^2(x,y)$ is
$Q$-periodic. Using~\eqref{tse:eq1} and the
expression~\eqref{def:A0star}, we then obtain
that $\overline{u}^\star_1$ satisfies
\begin{equation}
\label{PB:u1barstar_pre}
\left\{ 
\begin{array}{l l}
-\hbox{div}\left[A^\star_{per} \nabla \overline{u}_1^\star \right]
= 
\hbox{div}\left[\widetilde{B} \nabla u_0^\star \right] 
& \text{ in }\mathcal{D}, 
\\
\overline{u}_1^\star=0 & \text{ on }\partial\mathcal{D},
\end{array}
\right.
\end{equation}
with
\begin{equation}
\label{def:overlineB_pre}
\forall 1 \leq i,j \leq d, \quad
\widetilde{B}_{ij} = 
\int_Q e_i^T A_{per} \nabla \psi_{e_j}
+
\int_Q e_i^T B_{per} (e_j + \nabla \wper_{e_j}).
\end{equation}
Mutiplying~\eqref{eq:psi_p} (for $p=e_j$) by $\wper_{e_i}$ and
integrating over $Q$, we find that 
$$
\forall 1 \leq i,j \leq d, \quad
\int_Q (\nabla \wper_{e_i})^T A_{per} \nabla \psi_{e_j}
=
-\int_Q (\nabla \wper_{e_i})^T B_{per} (e_j + \nabla \wper_{e_j}).
$$
Inserting this relation in~\eqref{def:overlineB_pre}, we deduce that the
matrix $\widetilde{B}$ is equal to the matrix $\overline{B}$ defined
by~\eqref{def:overlineB}. We hence deduce from~\eqref{PB:u1barstar_pre}
that $\overline{u}^\star_1$ indeed satisfies~\eqref{PB:u1barstar}. 

These formal computations are formalized in a rigorous way in the
following Propositions:
\begin{prop}
\label{theo:u1bar0}
Assume that, for any $p \in \RR^d$, the corrector $\wper_p$ solution
to~\eqref{PB:w0} satisfies $\wper_p \in W^{1,\infty}(\RR^d)$, and that
the solution $u_0^\star$ to~\eqref{eq:homog-0} satisfies $u_0^\star \in
W^{2,\infty}({\cal D})$. 
Then the function $\overline{u}_1^\varepsilon$ solution
to~\eqref{PB:u1bareps} converges, weakly in $H^1(\mathcal{D})$ and
strongly in $L^2(\mathcal{D})$, to the unique solution
$\overline{u}_1^\star$ to~\eqref{PB:u1barstar}.
\end{prop}

The regularity assumptions on $\wper_p$ and $u_0^\star$ ensure that
$\nabla u_0^\eps$ in the right-hand side of~\eqref{PB:u1bareps} can be
controlled in the appropriate norm. 

\begin{prop}
\label{theo:restubar1}
Let $\overline{u}_1^\varepsilon$ be the solution to~\eqref{PB:u1bareps},
$\overline{u}_1^\star$ be the solution to~\eqref{PB:u1barstar}
and $u_0^\star$ be the solution to~\eqref{eq:homog-0}. For any $p \in
\RR^d$, let $\wper_p$ be the solution to~\eqref{PB:w0} and
$\psi_p$ be the solution to~\eqref{eq:psi_p}.

Introduce $\overline{v}_1^\varepsilon$ defined by
$$
\overline{v}_1^\varepsilon = \overline{u}_1^\star + \varepsilon
\sum\limits_{i=1}^d \left( \wper_{e_i}\left(\frac{\cdot}{\varepsilon}\right)
  \partial_i \overline{u}_1^\star +
  \psi_{e_i}\left(\frac{\cdot}{\varepsilon}\right) \partial_i u_0^\star
\right),
$$
and assume that $u_0^\star \in W^{2,\infty}(\mathcal{D})$,
$\overline{u}_1^\star \in W^{2,\infty}(\mathcal{D})$, and that, for any
$p \in \RR^d$, we have
$\wper_p \in W^{1,\infty}(\RR^d)$ and
$\psi_p \in W^{1,\infty}(\RR^d)$. 
We then have
$$
\|\overline{u}_1^\varepsilon -
\overline{v}_1^\varepsilon\|_{H^1(\mathcal{D})} 
\leq C \sqrt{\varepsilon} 
$$
for a constant $C$ independent of $\varepsilon$. 
\end{prop}

Again, under assumptions~\eqref{hyp:a_holder}
and~\eqref{hyp:b_holder}, we have  
$\wper_p \in W^{1,\infty}(\RR^d)$ and $\psi_p \in W^{1,\infty}(\RR^d)$
for any $p \in \RR^d$ (see
e.g.~\cite[Theorem 8.22 and Corollary~8.36]{gilbarg-trudinger}). 

\bigskip

We finally turn to $\phi_k^\varepsilon$ solution to~\eqref{PB:phikeps},
namely 
$$
\left\{ 
\begin{array}{l l}
\dps
-\hbox{div} \left[ A_{per} \left(\frac{\cdot}{\varepsilon}\right) 
\nabla \phi_k^\varepsilon \right]
=
\hbox{div} \left[ c_k^\eps \right] 
& \text{ in $\mathcal{D}$}, 
\\
\phi_k^\varepsilon=0 & \text{ on $\partial\mathcal{D}$},
\end{array}
\right.
$$
with
$$
c_k^\eps(x) = \mathbf{1}_{Q+k}\left(\frac{\cdot}{\varepsilon}\right)
  B_{per}\left(\frac{\cdot}{\varepsilon}\right) \nabla u_0^\varepsilon.
$$
Assume momentarily that the sequence $\nabla u_0^\varepsilon$ is bounded
in $L^\infty({\cal D})$ (we have proved such a bound above,
see~\eqref{eq:bound_u_dur}, under the strong
assumptions~\eqref{hyp:regul_f} and~\eqref{hyp:a_holder}). Then, for any
$k \in \ZZ^d$, $c_k^\eps$ converges to 0 in $L^2({\cal D})$.
Using the coercivity of $A_{per}$, this implies that
$\phi_k^\eps$ converges to 0 in $H^1({\cal D})$. We thus have the
following result, which will be rigourously proved in
Section~\ref{2scale:proofs} below:

\begin{prop}
\label{theo:phikeps}
Let $\phi_k^\varepsilon$ be the solution to~\eqref{PB:phikeps}, and let
$u_0^\star$ and $\wper_p$ be the solutions to~\eqref{eq:homog-0}
and~\eqref{PB:w0}. 
Assume that $u_0^\star \in W^{2,\infty}(\mathcal{D})$ and that, for any
$p \in \RR^d$, we have $\wper_p \in W^{1,\infty}(\RR^d)$. 
Then $\phi_k^\varepsilon$ converges to 0 in $H^1(\mathcal{D})$.
\end{prop}

To describe more precisely the behavior of $\phi_k^\varepsilon$, we 
need to introduce the auxilliary function $\chi_p$ defined by~\eqref{PB:chip}
below. Recall first that $Q=(-1/2,1/2)^d$.
Following the same arguments as in~\cite[Lemma~$4$]{Blanc2010}, we
have the following result, which will be useful in the sequel.

\begin{lemme}
\label{lem4}
For any $p\in \RR^d$, the problem
\begin{equation}
\label{PB:chip}
\left\{ 
\begin{array}{l l}
-\hbox{div}\left[A_{per} \nabla \chi_{p} \right]=
\hbox{div}\left[\mathbf{1}_Q B_{per}(p + \nabla \wper_p)\right] & \text{ in
  $\RR^d$}, 
\\
\chi_{p} \in L^2_{loc}(\RR^d), \quad 
\nabla \chi_{p} \in \left(L^2(\RR^d)\right)^d,
\end{array}
\right.
\end{equation}
has a solution which is unique up to the addition of a
constant. In addition, under assumption~\eqref{hyp:a_holder}, there
exists a solution of~\eqref{PB:chip} and a constant $C>0$ such that 
\begin{eqnarray}
\forall x \in \RR^d \text{ with }|x| \geq 1, \quad
|\nabla \chi_p| &\leq& \frac{C}{|x|^d},
\label{res:lem4}
\\
\forall x \in \RR^d, \quad 
|\chi_p| & \leq & \frac{C}{1+|x|^{d-1}}.
\label{res-b:lem4}
\end{eqnarray}
In the sequel, we will always refer to that particular solution
of~\eqref{PB:chip}. 
\end{lemme}

We are now in position to make precise the behavior of
$\phi_k^\varepsilon$ in the $H^1$ norm. Let us first argue
formally. Introduce the matrix $E_k = \mathbf{1}_{Q+k} B_{per}$. Using the
periodicity of $A_{per}$, $B_{per}$ and $\wper_p$, and after changing
variables, we recast~\eqref{PB:chip} as
$$
-\hbox{div}\left[ A_{per}\left(\frac{\cdot}{\varepsilon}\right) 
\nabla \chi_p \left( \frac{\cdot}{\eps}-k \right) \right]
=
\hbox{div}\left[ E_k\left(\frac{\cdot}{\varepsilon}\right) 
\left( p + \nabla \wper_p\left(\frac{\cdot}{\varepsilon}\right) \right) \right].
$$
In turn, the problem~\eqref{PB:phikeps} reads
$$
-\hbox{div}\left[A_{per}\left(\frac{\cdot}{\varepsilon}\right)\nabla
  \phi_k^\varepsilon \right]
=
\hbox{div}\left[E_k\left(\frac{\cdot}{\varepsilon}\right)
  \nabla u_0^\varepsilon
\right] 
\approx
\sum_{i=1}^d 
\hbox{div}\left[ E_k\left(\frac{\cdot}{\varepsilon}\right)
  \partial_i u_0^\star \left(
e_i + \nabla \wper_{e_i} \left(\frac{\cdot}{\varepsilon}\right)
\right)
\right], 
$$
where we have used the expansion~\eqref{tse:u0eps} of $u_0^\eps$ (in
which we have only kept the highest order terms). Assuming that, in the
above equation, $x$ and $x/\eps$ are independent variables, we thus see
that 
$\dps
\nabla \phi_k^\varepsilon(x) \approx \sum_{i=1}^d 
\partial_i u_0^\star(x) \ \nabla \chi_{e_i} \left( \frac{x}{\eps}-k \right)
$,
and thus 
$\dps 
\phi_k^\varepsilon(x) \approx \eps \sum_{i=1}^d 
\partial_i u_0^\star(x) \ \chi_{e_i} \left( \frac{x}{\eps}-k
\right)
$.
These formal manipulations motivate the following result, the rigorous
proof of which is postponed until Section~\ref{2scale:proofs}:
\begin{prop}
\label{theo:restu1}
Let $\phi_k^\varepsilon$ be the solution to~\eqref{PB:phikeps} and
$\chi_{e_i}$ be the solution to~\eqref{PB:chip}, for $1 \leq i \leq
d$. Introduce  
\begin{equation}
\label{def:Ieps}
I_\varepsilon = \left\{ k \in \ZZ^d \text{ such that } \varepsilon(Q+k)
  \cap \mathcal{D} \ne \emptyset \right\}, 
\quad 
\text{Card}(I_\varepsilon) \sim \varepsilon^{-d},
\end{equation}
and 
$$
\betaa_k^\varepsilon = \varepsilon \sum\limits_{i=1}^d 
\chi_{e_i}\left(\frac{\cdot}{\varepsilon}-k\right) \partial_i u_0^\star,
$$
where $u_0^\star$ is solution to~\eqref{eq:homog-0}. 
Assume that $u_0^\star \in W^{2,\infty}(\mathcal{D})$, and
that~\eqref{hyp:a_holder} holds. 
We then have
$$
\sum\limits_{k \in I_\varepsilon}
\|\phi_k^\varepsilon - \betaa_k^\varepsilon\|^2_{H^1(\mathcal{D})} 
\leq C \varepsilon \ln(1/\varepsilon),
$$
where $C$ is a constant independent of $\varepsilon$.
\end{prop}

\section{Proofs of Theorems~\ref{exp2scaleueta} and~\ref{exp2scaleueta-1D}}
\label{sec:preuve_main}

\begin{proof}[Proof of Theorem~\ref{exp2scaleueta}]
We have shown above (see~\eqref{eq:decompo_gene}) that
$$
u^\varepsilon_\eta(x,\omega) = u_0^\varepsilon(x) + \eta \esp(X_0) \overline{u}_1^\varepsilon(x) + \eta \sum\limits_{k \in I_\varepsilon} (X_k(\omega)-\esp(X_0)) \phi_k^\varepsilon(x) + \eta^2 r_\eta^\varepsilon(x,\omega),
$$
where the set $I_\eps$ is defined by~\eqref{def:Ieps} (recall that
$\phi_k^\eps \equiv 0$ whenever $k \in \ZZ^d$ is such that $k \notin
I_\eps$). 
Using the fact that $X_k$ are i.i.d. scalar random variables, we have
$$
\esp \left[ \|u^\varepsilon_\eta -
  v^\varepsilon_\eta\|^2_{H^1(\mathcal{D})} \right] 
\leq C \left[ D_0^2 + D_1^2 + D_2^2 + D_3^2 \right],
$$
where
\begin{eqnarray*}
D_0 &=& 
\left\|u_0^\varepsilon - u_0^\star - \varepsilon \sum\limits_{p=1}^d
  \wper_{e_p}\left(\frac{\cdot}{\varepsilon}\right) \partial_p u_0^\star
\right\|_{H^1(\mathcal{D})}, 
\\
D_1 &=& \eta |\esp(X_0)| \left\| \overline{u}_1^\varepsilon -
  \overline{u}_1^\star - \varepsilon \sum\limits_{p=1}^d
  \left(\wper_{e_p}\left(\frac{\cdot}{\varepsilon}\right) \partial_p
    \overline{u}_1^\star + \psi_{e_p}\left(\frac{\cdot}{\varepsilon}\right)
    \partial_p u_0^\star \right)\right\|_{H^1(\mathcal{D})}, 
\\
D_2 &=& \eta \sqrt{\var(X_0)} \sqrt{\sum\limits_{k \in I_\varepsilon}
  \left\|\phi_k^\varepsilon - \varepsilon \sum\limits_{p=1}^d
    \chi_{e_p}\left(\frac{\cdot}{\varepsilon}-k\right) \partial_p u_0^\star
  \right\|_{H^1(\mathcal{D})}^2},
\\
D_3 &=& \eta^2 \sqrt{ \esp \left[ \| r_\eta^\varepsilon
    \|^2_{H^1(\mathcal{D})} \right] }.
\end{eqnarray*}
We have shown in Propositions~\ref{the:u0eps},~\ref{theo:restubar1}
and~\ref{theo:restu1} that $D_0 \leq C \sqrt{\varepsilon}$, 
$D_1 \leq C \eta \sqrt{\varepsilon}$ and 
$D_2 \leq C \eta \sqrt{\varepsilon \ln(1/\varepsilon)}$
respectively, for a constant $C$ independent of
$\eps$ and $\eta$ (note that all assumptions of these propositions are
satisfied since, in view of~\eqref{hyp:a_holder}
and~\eqref{hyp:b_holder}, we have $\wper_p \in W^{1,\infty}(\RR^d)$ and
$\psi_p \in W^{1,\infty}(\RR^d)$ for any $p \in \RR^d$).
Next, using Lemma~\ref{lem1eta}, we see that $D_3
\leq C \eta^2$ for a constant $C$ independent of $\eps$ and $\eta$. This
concludes the proof of~\eqref{restueta}.
\end{proof}

\medskip

\begin{proof}[Proof of Theorem~\ref{exp2scaleueta-1D}]
To fix the idea, we choose ${\cal D} = (0,1)$.
We again argue on the basis of~\eqref{eq:decompo_gene}. 
Tedious but straightforward computations show that, in dimension one,
the estimates of Propositions~\ref{the:u0eps},~\ref{theo:restubar1}
and~\ref{theo:restu1} read
$$
\left\| \eps \frac{d \theta^\eps_0}{dx} \right\|_{L^2(0,1)} 
\leq
C \eps,
\quad \quad
\left\| \frac{d\overline{u}_1^\varepsilon}{dx} - 
\frac{d\overline{v}_1^\varepsilon}{dx} \right\|_{L^2(0,1)} 
\leq 
C \varepsilon, 
\quad \quad
\sum\limits_{k \in I_\varepsilon}
\left\| \frac{d\phi_k^\varepsilon}{dx} - 
\frac{d\betaa_k^\varepsilon}{dx} \right\|^2_{L^2(0,1)} 
\leq 
C \varepsilon.
$$
We thus obtain
\begin{equation}
\label{1D:1}
\sqrt{ \esp \left[ 
\left\| \frac{du^\varepsilon_\eta}{dx} - 
\frac{dv^\varepsilon_\eta}{dx} \right\|^2_{L^2(0,1)} 
\right] }
\leq C \left( \varepsilon + \eta \sqrt{\varepsilon} + \eta^2 \right).
\end{equation}
We next write that, almost surely,
\begin{equation}
\label{1D:2}
\left\| u^\varepsilon_\eta(\cdot,\omega) - 
v^\varepsilon_\eta(\cdot,\omega) \right\|_{L^\infty(0,1)} 
\leq 
\left\| \frac{du^\varepsilon_\eta}{dx}(\cdot,\omega) - 
\frac{dv^\varepsilon_\eta}{dx}(\cdot,\omega) \right\|_{L^2(0,1)} 
+
\left| u^\varepsilon_\eta(0,\omega) - 
v^\varepsilon_\eta(0,\omega) \right|.
\end{equation}
Using that $u^\varepsilon_\eta(0,\omega) = u_0^\star(0) =
\overline{u}_1^\star(0) = 0$ and that 
$\wper$ and $\psi$ belong to $L^\infty(\RR)$, we obtain that
$$
\left| u^\varepsilon_\eta(0,\omega) - 
v^\varepsilon_\eta(0,\omega) \right| \leq C \eps + C \eps \eta
\left| (u_0^\star)'(0) \sum_{k \in I_\varepsilon} (X_k(\omega)-\esp(X_0))
  \chi\left(-k\right) \right|,
$$
hence, using that $\chi \in L^\infty(\RR)$, we have
$$
\esp \left[ \left| u^\varepsilon_\eta(0,\omega) - 
v^\varepsilon_\eta(0,\omega) \right|^2 \right] 
\leq 
C \eps^2 + C \var(X_0) \eps^2 \eta^2
\sum_{k \in I_\varepsilon} \chi^2\left(-k\right)
\leq
C \eps^2 + C \eta^2 \eps.
$$
Collecting this result with~\eqref{1D:1} and~\eqref{1D:2} yields the
bound~\eqref{restueta-1d-Linfty}. Likewise,
collecting~\eqref{restueta-1d-Linfty} 
and~\eqref{1D:1}, we obtain the bound~\eqref{restueta-1d-H1}. 
This concludes the proof of Theorem~\ref{exp2scaleueta-1D}.
\end{proof}

\section{Proofs of the two scale expansions}
\label{2scale:proofs}

We collect in this section the proofs of the results stated in
Section~\ref{2scale}. The following technical result, already present
in~\cite[p.~27]{Jikov1994}, and that we recall here for the sake of
completeness, will be useful. 

\begin{lemme}
\label{lem3}
Let $\mathcal{D}$ be a bounded open set of $\RR^d$. Consider $Z \in
(L^2_{loc}(\RR^d))^d$ a $Q$-periodic vector field such that 
$$
\hbox{div }(Z)=0 \quad \text{and} \quad \int_Q Z = 0.
$$
Then, for any $v \in W^{1,\infty}(\mathcal{D})$, we have
$$
\left\| \hbox{div}\left[Z\left(\frac{\cdot}{\varepsilon}\right) v \right]
\right\|_{H^{-1}(\mathcal{D})} \leq C \varepsilon 
\left\| \nabla v \right\|_{L^\infty(\mathcal{D})}, 
$$
where $C$ is a constant independent of $\varepsilon$ and $v$.
\end{lemme}

Note that, as $Z$ is divergence free, we have
$$
\hbox{div}\left[Z\left(\frac{\cdot}{\varepsilon}\right) v \right]
=
Z\left(\frac{\cdot}{\varepsilon}\right) \cdot \nabla v.
$$
Since $Z$ is $Q$-periodic, this quantity converges weakly in
$L^2(\mathcal{D})$ to $\langle Z \rangle \cdot \nabla v = 0$, as the
average of $Z$ vanishes. The above result hence shows that, in the
$H^{-1}(\mathcal{D})$ norm, the above quantity vanishes at the rate
$\eps$.

\begin{proof}
In view of the assumptions of $Z$, there exists
(see~\cite[p. 6]{Jikov1994}) a skew symmetric matrix $J$ such that, 
$$
\forall 1 \leq j \leq d, \quad 
Z_j = \sum_{i=1}^d \frac{\partial J_{ij}}{\partial x_i}
$$
and
$$
\forall 1 \leq i,j \leq d, \quad J_{ij} \in H^1_{loc}(\RR^d),
\quad 
\text{$J_{ij}$ is $Q$-periodic},
\quad
\int_Q J_{ij} = 0.
$$
The $j$-th coordinate of the vector $\dps
Z\left(\frac{\cdot}{\varepsilon}\right)v$ reads 
\begin{eqnarray*}
\left[Z\left(\frac{x}{\varepsilon}\right)v(x) \right]_j 
&=& \sum\limits_{i=1}^d 
\frac{\partial J_{ij}}{\partial x_i}\left(\frac{x}{\varepsilon}\right)v(x) 
\\
&=& \varepsilon \sum\limits_{i=1}^d 
\frac{\partial}{\partial
  x_i}\left(J_{ij}\left(\frac{x}{\varepsilon}\right)v(x)\right) 
- 
\varepsilon \sum\limits_{i=1}^d J_{ij}\left(\frac{x}{\varepsilon}\right)
\frac{\partial v}{\partial x_i}(x)
\\
&=& \eps \widetilde{B}_j(x) - \eps B_j(x),
\end{eqnarray*}
where
$$
B_j(x) = \sum\limits_{i=1}^d J_{ij} \left(\frac{x}{\varepsilon}\right)
\frac{\partial v}{\partial x_i}(x)
\quad \text{and} \quad
\widetilde{B}_j(x) = \sum\limits_{i=1}^d 
\frac{\partial}{\partial
  x_i}\left(J_{ij}\left(\frac{x}{\varepsilon}\right)v(x)\right). 
$$
The vector $\widetilde{B}(x)$ is divergence free as $J$ is
skew symmetric. For any $\phi \in H^1_0(\mathcal{D})$, we thus have
\begin{eqnarray*}
\left< \hbox{div}\left[Z\left(\frac{\cdot}{\varepsilon}\right)v\right] ,
  \phi \right> 
&=& 
-\varepsilon \left< \hbox{div}\left[B \right] , \phi \right> 
\\
&=& \varepsilon \int_\mathcal{D} B \cdot \nabla \phi \\
&=& \varepsilon \sum\limits_{i,j=1}^d \int_\mathcal{D} \partial_j \phi
\, 
J_{ij}\left(\frac{\cdot}{\varepsilon}\right) \partial_i v, 
\end{eqnarray*}
hence
\begin{eqnarray*}
\left| \left<
    \hbox{div}\left[Z\left(\frac{\cdot}{\varepsilon}\right)v\right] ,
    \phi \right> \right| 
&\leq &
\varepsilon 
\|\nabla v\|_{L^\infty(\mathcal{D})} \|\phi\|_{H^1(\mathcal{D})}
\sum\limits_{i,j=1}^d \left\|
  J_{ij}\left(\frac{\cdot}{\varepsilon}\right)
\right\|_{L^2(\mathcal{D})} 
\\
&\leq &
\varepsilon 
\|\nabla v\|_{L^\infty(\mathcal{D})} \|\phi\|_{H^1(\mathcal{D})}
\sum\limits_{i,j=1}^d \left\|
  J_{ij} \right\|_{L^2(Q)}.
\end{eqnarray*}
As the above bound holds for any $\phi \in H^1_0(\mathcal{D})$, we
deduce that there exists $C$ such that, for any $v \in
W^{1,\infty}(\mathcal{D})$ and any $\eps$, we have
$$
\left\| \hbox{div}\left[Z\left(\frac{\cdot}{\varepsilon}\right)v\right]
\right\|_{H^{-1}(\mathcal{D})} \leq C \varepsilon
\| \nabla v \|_{L^\infty(\mathcal{D})}. 
$$
This concludes the proof.
\end{proof}

\subsection{Two scale expansion of $\overline{u}_1^\varepsilon$}

In this section, we prove Propositions~\ref{theo:u1bar0}
and~\ref{theo:restubar1}. 

\begin{proof}[Proof of Proposition~\ref{theo:u1bar0}]
This homogenization result is proved using the method of oscillating
test functions~\cite{murat,tartar}. 
The variational formulation of~\eqref{PB:u1bareps} reads
\begin{equation}
\label{def:varu1bar}
\forall v \in H^1_0({\mathcal D}),
\quad
{\cal A}_\varepsilon(\overline{u}_1^\varepsilon,v) = - L_\varepsilon(v),
\end{equation}
where, for any $u$ and $v$ in $H^1_0({\mathcal D})$,
$$
{\cal A}_\varepsilon(u,v) = \int_\mathcal{D} \left(\nabla v\right)^T
A_{per}\left(\frac{\cdot}{\varepsilon}\right) \nabla u 
\quad \text{and} \quad 
L_\varepsilon(v) = \int_\mathcal{D} \left(\nabla v\right)^T B_{per}
\left(\frac{\cdot}{\varepsilon}\right) \nabla u_0^\varepsilon. 
$$
Using the coercivity of $A_{per}$, the boundedness of $B_{per}$
and~\eqref{ape-1}, and taking $v =
\overline{u}_1^\varepsilon$ as a function test in~\eqref{def:varu1bar},
we obtain that $\overline{u}_1^\varepsilon$ is bounded in
$H^1_0(\mathcal{D})$. Thus, using the Rellich Theorem, we deduce that there
exists $\overline{u}_1^\star \in H^1_0(\mathcal{D})$ such that, up to the extraction of a
subsequence, 
$$
\overline{u}_1^{\varepsilon} \text{ converges to $\overline{u}_1^\star$, 
weakly in $H_0^1(\mathcal{D})$ and strongly in $L^2(\mathcal{D})$.}
$$
For any function $\varphi \in \mathcal{C}^{\infty}_0(\mathcal{D})$,
define the test function
$$
v^\varepsilon = \varphi + \varepsilon \sum\limits_{i=1}^d \wper_{e_i}\left(\frac{\cdot}{\varepsilon}\right) \partial_i \varphi,
$$
which obviously belongs to $H^1_0(\mathcal{D})$.
In view of~\eqref{def:varu1bar}, we have 
\begin{equation}
\label{def:varu1bar-veps}
{\cal A}_\varepsilon(\overline{u}_1^\varepsilon,v^\varepsilon) = - L_\varepsilon(v^\varepsilon).
\end{equation}
We now expand both sides of~\eqref{def:varu1bar-veps} in powers of
$\varepsilon$: 
\begin{eqnarray}
{\cal A}_\varepsilon(\overline{u}_1^\varepsilon,v^\varepsilon) &=& 
{\cal A}_\varepsilon^0(\overline{u}_1^\varepsilon,\varphi) + 
\varepsilon {\cal A}_\varepsilon^1(\overline{u}_1^\varepsilon,\varphi), 
\label{Aeps-exp}\\
L_\varepsilon(v^\varepsilon) &=& L_\varepsilon^0(\varphi) + \varepsilon
L_\varepsilon^1(\varphi) 
\label{Leps-exp},
\end{eqnarray}
where
\begin{eqnarray*}
{\cal A}_\varepsilon^0(\overline{u}_1^\varepsilon,\varphi) 
&=& 
\int_\mathcal{D} \left(\nabla \varphi + \sum\limits_{i=1}^d \nabla
  \wper_{e_i}\left(\frac{\cdot}{\varepsilon}\right) \partial_i \varphi\right)^T
A_{per}\left(\frac{\cdot}{\varepsilon}\right) \nabla
\overline{u}_1^\varepsilon, 
\\
{\cal A}_\varepsilon^1(\overline{u}_1^\varepsilon,\varphi) 
&=& 
\int_\mathcal{D}  \sum\limits_{i=1}^d
\wper_{e_i}\left(\frac{\cdot}{\varepsilon}\right) \left(\nabla \partial_i
  \varphi\right)^T A_{per}\left(\frac{\cdot}{\varepsilon}\right) \nabla
\overline{u}_1^\varepsilon, 
\\
L_\varepsilon^0(\varphi) 
&=& 
\int_\mathcal{D} \left(\nabla \varphi + \sum\limits_{i=1}^d \nabla
  \wper_{e_i}\left(\frac{\cdot}{\varepsilon}\right) \partial_i \varphi\right)^T
B_{per}\left(\frac{\cdot}{\varepsilon}\right) \nabla u_0^\varepsilon, 
\\
L_\varepsilon^1(\varphi) 
&=& 
\int_\mathcal{D}  \sum\limits_{i=1}^d
\wper_{e_i}\left(\frac{\cdot}{\varepsilon}\right) \left(\nabla \partial_i
  \varphi\right)^T B_{per}\left(\frac{\cdot}{\varepsilon}\right) \nabla
u_0^\varepsilon. 
\end{eqnarray*}
We now successively study the limit of these four quantities as $\eps \to 0$.
Using~\eqref{ape-1}, the fact that $\wper_{e_i} \in W^{1,\infty}(\RR^d)$, that 
$\overline{u}_1^\varepsilon$ is bounded in $H^1(\mathcal{D})$ and the 
boundedness of $A_{per}$ and $B_{per}$, we obtain
\begin{equation}
|{\cal A}^1_\varepsilon(\overline{u}_1^\varepsilon,\varphi) | \leq C 
\quad \text{and} \quad
 |L^1_\varepsilon(\varphi) | \leq C,
\quad
\text{$C$ independent of $\varepsilon$}. 
\label{A1eps-bound}
\end{equation}

We now turn to $L_\varepsilon^0$. Using the two scale
expansion~\eqref{tse:u0eps} of $u_0^\varepsilon$, we see that
\begin{equation}
\label{res0}
L_\varepsilon^0(\varphi) = L_\varepsilon^{00}(\varphi) +
L_\varepsilon^{01}(\varphi) + L_\varepsilon^{02}(\varphi), 
\end{equation}
where
\begin{eqnarray*}
L_\varepsilon^{00}(\varphi) &=& \sum\limits_{i,j=1}^d \int_\mathcal{D}
\left(e_i+\nabla \wper_{e_i}\left(\frac{\cdot}{\varepsilon}\right)\right)^T
B_{per}\left(\frac{\cdot}{\varepsilon}\right)\left(e_j+\nabla
  \wper_{e_j}\left(\frac{\cdot}{\varepsilon}\right)\right)\partial_i \varphi \
\partial_j u_0^\star, 
\\
L_\varepsilon^{01}(\varphi) &=& \varepsilon \sum\limits_{i,j=1}^d
\int_\mathcal{D} \left(e_i+\nabla
  \wper_{e_i}\left(\frac{\cdot}{\varepsilon}\right)\right)^T
B_{per}\left(\frac{\cdot}{\varepsilon}\right)
\wper_{e_j}\left(\frac{\cdot}{\varepsilon}\right)\left(\nabla \partial_j
  u_0^\star\right) \ \partial_i \varphi,  
\\
L_\varepsilon^{02}(\varphi) &=& \sum\limits_{i=1}^d \int_\mathcal{D}
\left(e_i+\nabla \wper_{e_i}\left(\frac{\cdot}{\varepsilon}\right)\right)^T
B_{per}\left(\frac{\cdot}{\varepsilon}\right) \varepsilon \nabla
\theta_0^\varepsilon \ \partial_i \varphi.
\end{eqnarray*}
Using~\eqref{restepsu0}, $\wper_{e_i} \in W^{1,\infty}(\RR^d)$ and $u_0^\star \in
W^{2,\infty}(\mathcal{D})$, we obtain that 
\begin{equation}
\label{res1}
|L_\varepsilon^{02}(\varphi)| \leq C \|\varepsilon
\theta_0^\varepsilon\|_{H^1(\mathcal{D})} \leq C \sqrt{\varepsilon} 
\quad \text{and} \quad
|L_\varepsilon^{01}(\varphi)| \leq C \varepsilon,
\end{equation}
where $C$ is a constant independent of $\varepsilon$. Turning to
$L_\varepsilon^{00}$, we see, using that $B_{per}$ and $\wper_{e_i}$ are
$Q$-periodic, that 
$$
\left(e_i+\nabla \wper_{e_i}\left(\frac{\cdot}{\varepsilon}\right)\right)^T
B_{per}\left(\frac{\cdot}{\varepsilon}\right)\left(e_j+\nabla
  \wper_{e_j}\left(\frac{\cdot}{\varepsilon}\right)\right) 
\rightharpoonup 
\overline{B}_{ij}
\quad \text{weakly-$\star$ in $L^\infty$},
$$
where $\overline{B}$ is defined by~\eqref{def:overlineB}.
Thus
\begin{equation}
\label{res3}
L_\varepsilon^{00}(\varphi) \rightarrow \int_\mathcal{D} \left(\nabla
  \varphi\right)^T \overline{B} \nabla u_0^\star \quad \text{as
}\varepsilon \to 0. 
\end{equation}
Collecting~\eqref{res0},~\eqref{res1} and~\eqref{res3}, we
obtain that 
\begin{equation}
\label{convLeps}
L^0_\varepsilon(\varphi) \rightarrow \int_\mathcal{D} \left(\nabla
  \varphi\right)^T \overline{B} \nabla u_0^\star \quad \text{as
}\varepsilon \to 0. 
\end{equation}
We next turn to ${\cal A}^0_\varepsilon$. 
Using that $\dps
\hbox{div}\left[A_{per}\left(\frac{\cdot}{\varepsilon}\right)\left(e_i +
    \nabla \wper_{e_i}\left(\frac{\cdot}{\varepsilon}\right)
  \right)\right]=0$ and that $A_{per}$ is symmetric,
we obtain that 
\begin{equation}
\label{A0eps}
{\cal A}^0_\varepsilon(\overline{u}_1^\varepsilon,\varphi) = -\sum\limits_{i=1}^d \int_\mathcal{D} \overline{u}_1^\varepsilon \left(\nabla \partial_i \varphi\right)^T A_{per}\left(\frac{\cdot}{\varepsilon}\right)\left(e_i + \nabla \wper_{e_i}\left(\frac{\cdot}{\varepsilon}\right) \right) .
\end{equation}
Recall now that $\overline{u}_1^{\varepsilon} \to \overline{u}_1^\star$
strongly in $L^2(\mathcal{D})$ and that, 
as $A_{per}$ and $\wper_{e_i}$ are $Q$-periodic, we have
$$
A_{per}\left(\frac{\cdot}{\varepsilon}\right)\left(e_i + \nabla \wper_{e_i}\left(\frac{\cdot}{\varepsilon}\right) \right) \rightharpoonup \int_Q A_{per}\left(e_i + \nabla \wper_{e_i} \right) = A_{per}^\star e_i \quad \text{weakly-$\star$ in $L^\infty$},
$$
where $A_{per}^\star$ is defined by~\eqref{def:A0star}. We thus deduce
from~\eqref{A0eps} that 
$$
{\cal A}^0_\varepsilon(\overline{u}_1^\varepsilon,\varphi) \rightarrow
-\sum\limits_{i=1}^d \int_\mathcal{D} \overline{u}_1^\star \left(\nabla
  \partial_i \varphi\right)^T A_{per}^\star e_i \quad \text{as }\varepsilon
\to 0. 
$$
Collecting~\eqref{def:varu1bar-veps},~\eqref{Aeps-exp},~\eqref{Leps-exp},~\eqref{A1eps-bound}, 
the above limit and~\eqref{convLeps}, we obtain that
$\overline{u}_1^\star$ satisfies 
$$
-\sum\limits_{i=1}^d \int_\mathcal{D} \overline{u}_1^\star \left(\nabla
  \partial_i \varphi\right)^T A_{per}^\star e_i
=
- \int_\mathcal{D} \left(\nabla
  \varphi\right)^T \overline{B} \nabla u_0^\star
$$
for any $\varphi \in \mathcal{C}^{\infty}_0(\mathcal{D})$.
This shows that $\overline{u}_1^\star$ solves~\eqref{PB:u1barstar} (which has a unique solution) and thus concludes the proof of Proposition~\ref{theo:u1bar0}.
\end{proof}

\bigskip

\begin{proof}[Proof of Proposition~\ref{theo:restubar1}]

The proof mostly goes by using the coercivity of $A_{per}$ and showing
that, in some appropriate norm, $-\hbox{div}\left[A_{per}(\nabla
  \overline{u}_1^\varepsilon - \nabla \overline{v}_1^\varepsilon)
\right]$ is small. However, a technical difficulty comes from the fact
that $\overline{v}_1^\varepsilon \notin H^1_0(\mathcal{D})$, as it does
not vanish on $\partial \mathcal{D}$. A preliminary step (Step 1 below)
thus consists in approximating $\overline{v}_1^\varepsilon$ by a
function (namely $g_1^\varepsilon$ defined by~\eqref{eq:def_geps} below) that is
equal to $\overline{v}_1^\varepsilon$ away from the
boundary $\partial \mathcal{D}$, but vanishes on the boundary. Step 2
consists in estimating the difference $\overline{u}_1^\varepsilon -
g_1^\varepsilon$. 

\medskip

\noindent
\textbf{Step 1: Truncation of $\overline{v}_1^\varepsilon$}

Let us define $\tau_\varepsilon \in \mathcal{C}^\infty_0(\mathcal{D})$
such that $0 \leq \tau_\varepsilon(x) \leq 1$ for all $x \in
\mathcal{D}$, $\tau_\varepsilon(x) = 1$ when $\text{dist}(\partial
\mathcal{D}, x) \geq \varepsilon$ and $\varepsilon \| \nabla
\tau_\varepsilon \|_{L^\infty(\mathcal{D})} \leq C$, where $C$ is a
constant independent of $\varepsilon$. We denote by
$\mathcal{D}_\varepsilon \subset \mathcal{D}$ the set of
$\RR^d$ defined by
$$
\mathcal{D}_\varepsilon := \left\{ x \in \mathcal{D} \text{ such that }
\text{dist}(\partial \mathcal{D}, x) \geq \varepsilon \right\}
$$
and we note that
$$
|\mathcal{D} \setminus \mathcal{D}_\varepsilon| \leq C \eps.
$$
Introduce now $g_1^\varepsilon\in H^1_0(\mathcal{D})$ defined by 
\begin{equation}
\label{eq:def_geps}
g_1^\varepsilon = \overline{u}_1^\star + \varepsilon \tau_\varepsilon 
\sum\limits_{i=1}^d 
\left(\wper_{e_i}\left(\frac{\cdot}{\varepsilon}\right) \partial_i
  \overline{u}_1^\star 
+ \psi_{e_i}\left(\frac{\cdot}{\varepsilon}\right) \partial_i u_0^\star \right),
\end{equation}
where $\overline{u}_1^\star$ is the solution to~\eqref{PB:u1barstar}, 
$u_0^\star$ is the solution to~\eqref{eq:homog-0},
and $\wper_{e_i}$ and $\psi_{e_i}$ are solutions (with $p=e_i$)
to~\eqref{PB:w0} and~\eqref{eq:psi_p}, respectively.
Note that $g_1^\varepsilon = \overline{v}_1^\varepsilon$ except in a
neighboorhood of
$\partial \mathcal{D}$. In the sequel, we estimate
$\overline{v}_1^\varepsilon-g_1^\varepsilon$. In the next Step, we
estimate $g_1^\varepsilon-\overline{u}_1^\varepsilon$. 

By definition, 
\begin{equation}
\nabla \overline{v}_1^\varepsilon - \nabla g_1^\varepsilon = 
e_0^\varepsilon - e_1^\varepsilon + \varepsilon e_2^\varepsilon, 
\label{eq1:restubar1}
\end{equation}
where
\begin{eqnarray*}
e_0^\varepsilon &=& (1-\tau_\varepsilon) \sum\limits_{i=1}^d
\left(\nabla \wper_{e_i}\left(\frac{\cdot}{\varepsilon}\right) \partial_i
  \overline{u}_1^\star + \nabla
  \psi_{e_i}\left(\frac{\cdot}{\varepsilon}\right) \partial_i u_0^\star
\right), 
\\
e_1^\varepsilon &=& \varepsilon \nabla \tau_\varepsilon
\sum\limits_{i=1}^d \left(\wper_{e_i}\left(\frac{\cdot}{\varepsilon}\right)
  \partial_i \overline{u}_1^\star +
  \psi_{e_i}\left(\frac{\cdot}{\varepsilon}\right) \partial_i u_0^\star
\right), 
\\
e_2^\varepsilon &=& (1-\tau_\varepsilon) \sum\limits_{i=1}^d \left(
  \wper_{e_i}\left(\frac{\cdot}{\varepsilon}\right) \nabla(\partial_i
  \overline{u}_1^\star) +  \psi_{e_i}\left(\frac{\cdot}{\varepsilon}\right)
  \nabla(\partial_i u_0^\star) \right). 
\end{eqnarray*}
We now bound from above successively the $L^2$ norm of
$e_2^\varepsilon$, $e_1^\varepsilon$ and $e_0^\varepsilon$. First, as
$u_0^\star \in W^{2,\infty}(\mathcal{D})$, $\overline{u}_1^\star \in
W^{2,\infty}(\mathcal{D})$, $\psi_{e_i} \in W^{1,\infty}(\RR^d)$, $\wper_{e_i} \in
W^{1,\infty}(\RR^d)$ and $0 \leq \tau_\eps \leq 1$, we have
\begin{equation}
\|e_2^\varepsilon\|^2_{L^2(\mathcal{D})}
\leq C, \quad \text{$C$ independent of $\eps$}. 
\label{eq2:restubar1}
\end{equation}
The same arguments lead to 
\begin{eqnarray}
\|e_1^\varepsilon\|^2_{L^2(\mathcal{D})} 
&=& 
\int_{\mathcal{D}}
\left[\sum\limits_{i=1}^d
  \left(\wper_{e_i}\left(\frac{\cdot}{\varepsilon}\right) \partial_i
    \overline{u}_1^\star + \psi_{e_i}\left(\frac{\cdot}{\varepsilon}\right)
    \partial_i u_0^\star \right) \right]^2 
\left| \eps \nabla \tau_\eps \right|^2
\nonumber \\
&\leq& C \ |\mathcal{D} \setminus \mathcal{D}_\varepsilon| 
\nonumber \\
 &\leq& C \varepsilon, 
\label{eq3:restubar1}
\end{eqnarray}
for a constant $C$ independent of $\varepsilon$. 
We next write
\begin{equation}
\|e_0^\varepsilon\|^2_{L^2(\mathcal{D})} 
\leq
|\mathcal{D} \setminus \mathcal{D}_\varepsilon| \left\| \
\sum\limits_{i=1}^d
\left(\nabla \wper_{e_i}\left(\frac{\cdot}{\varepsilon}\right) \partial_i
  \overline{u}_1^\star + \nabla
  \psi_{e_i}\left(\frac{\cdot}{\varepsilon}\right) \partial_i u_0^\star
\right)
\right\|^2_{L^\infty(\mathcal{D})} 
\leq C \varepsilon. 
\label{eq4:restubar1}
\end{equation}
Collecting~\eqref{eq1:restubar1},~\eqref{eq2:restubar1},~\eqref{eq3:restubar1}
and~\eqref{eq4:restubar1}, we have 
$$
\|\nabla \overline{v}_1^\varepsilon - 
\nabla g_1^\varepsilon\|^2_{L^2(\mathcal{D})} \leq C \varepsilon,
\quad
\text{$C$ independent of $\varepsilon$.}
$$
Observing that
$$
\|\overline{v}_1^\varepsilon - g_1^\varepsilon\|^2_{L^2(\mathcal{D})} 
\leq 
2d \varepsilon^2 \sum\limits_{i=1}^d 
\left(\| \wper_{e_i}\|_{L^\infty}^2
  \|\overline{u}_1^\star\|^2_{H^1(\mathcal{D})} +
  \|\psi_{e_i}\|_{L^\infty}^2 \|u_0^\star\|^2_{H^1(\mathcal{D})} \right) 
\leq C \varepsilon^2,
$$
we obtain that
\begin{equation}
\|\overline{v}_1^\varepsilon - g_1^\varepsilon\|_{H^1(\mathcal{D})} \leq
C \sqrt{\varepsilon},
\quad
\text{$C$ independent of $\varepsilon$.}
\label{eq11:restubar1}
\end{equation}

\medskip

\noindent
{\bf Step 2:}
We next turn to estimating $\overline{u}_1^\varepsilon -
g_1^\varepsilon$. Using that $A_{per}$ is coercive and the fact that
$\overline{u}_1^\varepsilon - g_1^\varepsilon \in H^1_0(\mathcal{D})$,
we have 
\begin{eqnarray}
\alpha \|\overline{u}_1^\varepsilon -
g_1^\varepsilon\|^2_{H^1(\mathcal{D})} 
&\leq& 
\int_\mathcal{D} \left(\nabla \overline{u}_1^\varepsilon - \nabla
  g_1^\varepsilon \right)^T
A_{per}\left(\frac{\cdot}{\varepsilon}\right) \left(\nabla
  \overline{u}_1^\varepsilon - \nabla g_1^\varepsilon \right) 
\nonumber \\
&\leq& 
\int_\mathcal{D} \left(\nabla \overline{u}_1^\varepsilon - \nabla
  g_1^\varepsilon \right)^T
A_{per}\left(\frac{\cdot}{\varepsilon}\right) \left(\nabla
  \overline{u}_1^\varepsilon - \nabla \overline{v}_1^\varepsilon \right)
+ \int_\mathcal{D} \left(\nabla \overline{u}_1^\varepsilon - \nabla
  g_1^\varepsilon \right)^T
A_{per}\left(\frac{\cdot}{\varepsilon}\right) \left(\nabla
  \overline{v}_1^\varepsilon - \nabla g_1^\varepsilon \right) 
\nonumber \\
&\leq& 
\|\overline{u}_1^\varepsilon - g_1^\varepsilon\|_{H^1(\mathcal{D})}
\left( \left\|\hbox{div}\left[
      A_{per}\left(\frac{\cdot}{\varepsilon}\right) \left(\nabla
        \overline{u}_1^\varepsilon - \nabla \overline{v}_1^\varepsilon
      \right)\right] \right\|_{H^{-1}(\mathcal{D})} +
  \|A_{per}\|_{L^\infty} \|\overline{v}_1^\varepsilon -
  g_1^\varepsilon\|_{H^1(\mathcal{D})} \right),
\label{eq5:restubar1}
\end{eqnarray}
where the constant $\alpha > 0$ only depends on the coercivity constant
of $A_{per}$ and the Poincar\'e constant of the domain ${\cal D}$. 
In the sequel, we bound from above $\dps \left\|\hbox{div}\left[
      A_{per}\left(\frac{\cdot}{\varepsilon}\right) \left(\nabla
        \overline{u}_1^\varepsilon - \nabla \overline{v}_1^\varepsilon
      \right)\right] \right\|_{H^{-1}(\mathcal{D})}$.

By definition of $\overline{v}_1^\varepsilon$, we have
$$
\overline{v}_1^\varepsilon = \widehat{v}_1^\varepsilon + 
\widetilde{v}_1^\varepsilon,
$$
with
$$
\widehat{v}_1^\varepsilon 
= 
\overline{u}_1^\star + \varepsilon \sum\limits_{i=1}^d
\wper_{e_i}\left(\frac{\cdot}{\varepsilon}\right) \partial_i
\overline{u}_1^\star
\quad \text{and} \quad
\widetilde{v}_1^\varepsilon = \varepsilon \sum\limits_{i=1}^d 
\psi_{e_i}\left(\frac{\cdot}{\varepsilon}\right) \partial_i u_0^\star.
$$
Using the equation~\eqref{PB:u1bareps} on $\overline{u}_1^\varepsilon$
and the relation~\eqref{PB:u1barstar} between $\overline{u}_1^\star$ and
$u_0^\star$, we compute 
\begin{eqnarray}
\hbox{div}\left[ A_{per}\left(\frac{\cdot}{\varepsilon}\right)
  \left(\nabla \overline{v}_1^\varepsilon-\nabla
    \overline{u}_1^\varepsilon \right)\right] 
&=&  
\hbox{div}\left[ A_{per}\left(\frac{\cdot}{\varepsilon}\right)\nabla
  \widehat{v}_1^\varepsilon - A^\star_{per} \nabla \overline{u}_1^\star
\right] +
\hbox{div}\left[A_{per}\left(\frac{\cdot}{\varepsilon}\right)\nabla
  \widetilde{v}_1^\varepsilon +
  B_{per}\left(\frac{\cdot}{\varepsilon}\right)\nabla u_0^\varepsilon -
  \overline{B}\nabla u_0^\star\right] 
\nonumber \\
&=& D_0 + D_1 + \varepsilon D_2, 
\label{eq6:restubar1}
\end{eqnarray}
where 
\begin{eqnarray*}
D_0 &=& \sum\limits_{i=1}^d
\hbox{div}\left(\left[A_{per}\left(\frac{\cdot}{\varepsilon}\right)\left(e_i
      + \nabla \wper_{e_i}\left(\frac{\cdot}{\varepsilon}\right)\right) -
    A_{per}^\star e_i \right] \partial_i \overline{u}^\star_1 \right), 
\\
D_1 &=& \sum\limits_{i=1}^d
\hbox{div}\left(\left[A_{per}\left(\frac{\cdot}{\varepsilon}\right)
    \nabla \psi_{e_i}\left(\frac{\cdot}{\varepsilon}\right) +
    B_{per}\left(\frac{\cdot}{\varepsilon}\right)\left(e_i + \nabla
      \wper_{e_i}\left(\frac{\cdot}{\varepsilon}\right)\right) -
    \overline{B} e_i \right] \partial_i u^\star_0 \right), 
\\
D_2 &=& \hbox{div} \left[ B_{per}\left(\frac{\cdot}{\varepsilon}\right)
\nabla \theta_0^\varepsilon \right]
+
\sum\limits_{i=1}^d
\hbox{div}\left[A_{per}\left(\frac{\cdot}{\varepsilon}\right) \left(
\wper_{e_i}\left(\frac{\cdot}{\varepsilon}\right) \nabla \partial_i
\overline{u}_1^\star
+
\psi_{e_i}\left(\frac{\cdot}{\varepsilon}\right) \nabla \partial_i u_0^\star 
\right) 
+
B_{per}\left(\frac{\cdot}{\varepsilon}\right)
\wper_{e_i}\left(\frac{\cdot}{\varepsilon}\right)
\nabla \partial_i u_0^\star \right].
\end{eqnarray*}
We now bound from above these three quantities.
As $A_{per}$ and $B_{per}$ are bounded, we see that
$$
\|D_2\|_{H^{-1}(\mathcal{D})} \leq 
C \|\theta_0^\varepsilon\|_{H^1(\mathcal{D})}
+
C \sum\limits_{i=1}^d 
\left[ \| \wper_{e_i}\|_{L^\infty} \|\overline{u}_1^\star\|_{H^2(\mathcal{D})} +
  \|\psi_{e_i}\|_{L^\infty} \|u_0^\star\|_{H^2(\mathcal{D})} +
  \| \wper_{e_i}\|_{L^\infty} \|u_0^\star\|_{H^2(\mathcal{D})} \right], 
$$
from which we infer, in view of~\eqref{restepsu0}, that
\begin{equation}
\varepsilon \|D_2\|_{H^{-1}(\mathcal{D})} \leq C \sqrt{\varepsilon},
\quad
\text{$C$ independent of $\varepsilon$.}
\label{eq7:restubar1}
\end{equation}
Let us now turn to $D_0$.
Consider, for any $1 \leq i \leq d$, the vector-valued function
$$
Z(y)=A_{per}(y)\left(e_i + \nabla \wper_{e_i}(y)\right) - A_{per}^\star e_i.
$$
We observe that $Z \in (L^2_{loc}(\RR^d))^d$ is divergence free,
$Q$-periodic and of vanishing mean. Since $\partial_i
\overline{u}_1^\star \in W^{1,\infty}(\mathcal{D})$, we can use
Lemma~\ref{lem3}, and we obtain
\begin{equation}
\|D_0\|_{H^{-1}(\mathcal{D})} \leq C \varepsilon,
\quad
\text{$C$ independent of $\varepsilon$.}
\label{eq8:restubar1}
\end{equation}
Turning now to $D_1$, we likewise consider, for any $1 \leq i \leq d$, 
the vector-valued function 
$$
\overline{Z}(y) = A_{per}(y) \nabla \psi_{e_i}(y) +  B_{per}(y)\left(e_i +
  \nabla \wper_{e_i}(y)\right) - \overline{B} e_i. 
$$
By construction, $\overline{Z} \in (L^2_{loc}(\RR^d))^d$ is $Q$-periodic
and divergence free, in 
view of the definition~\eqref{eq:psi_p} of $\psi_{e_i}$. In addition, the mean
of $\overline{Z}$ vanishes. Indeed, for any $1 \leq j \leq d$,
using~\eqref{def:overlineB},~\eqref{eq:psi_p}, the symmetry of
$A_{per}$ and~\eqref{PB:w0}, we have 
\begin{eqnarray*}
\int_Q \overline{Z} \cdot e_j 
&=& 
\int_Q e_j^T A_{per} \nabla \psi_{e_i} +  \int_Q e_j^T B_{per}\left(e_i +
\nabla \wper_{e_i}\right) - \int_Q (e_j+\nabla \wper_{e_j})^T B_{per}\left(e_i
  + \nabla \wper_{e_i}\right) 
\\
&=& 
\int_Q e_j^T A_{per} \nabla \psi_{e_i} -  \int_Q (\nabla \wper_{e_j})^T
B_{per}\left(e_i + \nabla \wper_{e_i}\right) 
\\
&=& \int_Q e_j^T A_{per} \nabla \psi_{e_i} +  \int_Q (\nabla \wper_{e_j})^T A_{per}
\nabla \psi_{e_i} 
\\
&=& \int_Q (\nabla \psi_{e_i})^T A_{per} (e_j+\nabla \wper_{e_j}) \\
&=& 0.
\end{eqnarray*}
Since $\partial_i u_0^\star \in W^{1,\infty}(\mathcal{D})$, we have that
$\overline{Z}$ and $\partial_i u_0^\star$ satisfy
the assumptions of Lemma~\ref{lem3}, hence
\begin{equation}
\|D_1\|_{H^{-1}(\mathcal{D})} \leq C \varepsilon,
\quad
\text{$C$ independent of $\varepsilon$.}
\label{eq9:restubar1}
\end{equation}

Collecting~\eqref{eq6:restubar1},~\eqref{eq7:restubar1},~\eqref{eq8:restubar1}
and~\eqref{eq9:restubar1}, we have 
\begin{equation}
\left\|\hbox{div}\left[ A_{per}\left(\frac{\cdot}{\varepsilon}\right)
    \left(\nabla \overline{u}_1^\varepsilon - \nabla
      \overline{v}_1^\varepsilon \right)\right]
\right\|_{H^{-1}(\mathcal{D})} \leq C \sqrt{\varepsilon},
\label{eq10:restubar1}
\end{equation}
where $C$ is a constant independent of $\varepsilon$. 
We now infer from~\eqref{eq11:restubar1},~\eqref{eq5:restubar1}
and~\eqref{eq10:restubar1} that 
$$
\alpha \|\overline{u}_1^\varepsilon -
g_1^\varepsilon\|^2_{H^1(\mathcal{D})} 
\leq
C \|\overline{u}_1^\varepsilon - g_1^\varepsilon\|_{H^1(\mathcal{D})}
\ \sqrt{\eps},
$$
hence
$$
\|\overline{u}_1^\varepsilon - g_1^\varepsilon\|_{H^1(\mathcal{D})} 
\leq
C \sqrt{\eps},
\quad \text{$C$ independent of $\eps$.}
$$

\medskip

\noindent
\textbf{Step 3: Conclusion}

Collecting the above bound with~\eqref{eq11:restubar1}, we deduce that
$$
\|\overline{u}_1^\varepsilon -
\overline{v}_1^\varepsilon\|_{H^1(\mathcal{D})} \leq C
\sqrt{\varepsilon},
\quad
\text{$C$ independent of $\varepsilon$.}
$$
We thus have
proved the claimed bound, and this concludes the
proof of Proposition~\ref{theo:restubar1}. 
\end{proof}

\subsection{Two-scale expansion of $\phi_k^\varepsilon$}

In this section, we prove Propositions~\ref{theo:phikeps}
and~\ref{theo:restu1}. 

\begin{proof}[Proof of Proposition~\ref{theo:phikeps}]

Introducing
$$
c_k^\varepsilon(x)
=
\mathbf{1}_{Q+k}\left(\frac{x}{\varepsilon}\right) 
B_{per}\left(\frac{x}{\varepsilon}\right) \nabla u_0^\varepsilon(x),
$$
the problem~\eqref{PB:phikeps} writes
$$
\left\{ 
\begin{array}{l l}
\dps
-\hbox{div} \left[ A_{per} \left(\frac{\cdot}{\varepsilon}\right) 
\nabla \phi_k^\varepsilon \right]
=
\hbox{div} \left[ c_k^\eps \right] 
& \text{ in $\mathcal{D}$}, 
\\
\phi_k^\varepsilon=0 & \text{ on $\partial\mathcal{D}$}.
\end{array}
\right.
$$
Multiplying this equation by $\phi_k^\varepsilon$, integrating over
${\cal D}$, and using the coercivity of $A_{per}$, we obtain that there
exists $C$ independent of $k$ and $\eps$ such that
\begin{equation}
\label{PB:phikeps_bis} 
\| \phi_k^\varepsilon \|_{H^1(\mathcal{D})} 
\leq
C \| c_k^\varepsilon \|_{L^2(\mathcal{D})}. 
\end{equation}
Let us now show that $c_k^\eps$ converges to 0 in $L^2({\cal D})$.
Using the expansion~\eqref{tse:u0eps}, we write
$$
\nabla u_0^\varepsilon = 
T^\eps + \varepsilon \nabla \theta_0^\varepsilon,
$$
with 
$$
T^\eps = 
\sum\limits_{i=1}^d \partial_i u_0^\star \ 
\left( e_i + \nabla \wper_{e_i}\left(\frac{\cdot}{\varepsilon}\right)
\right) 
+ \varepsilon \sum\limits_{i=1}^d \nabla (\partial_i u_0^\star) 
\wper_{e_i}\left(\frac{\cdot}{\varepsilon}\right). 
$$
Using the fact that $\wper_p \in W^{1,\infty}(\RR^d)$ and $u_0^\star \in
W^{2,\infty}(\mathcal{D})$, 
we see that $T^\eps$ is bounded in $L^\infty(\mathcal{D})$. We next
write 
\begin{eqnarray*}
\| c_k^\varepsilon \|^2_{L^2(\mathcal{D})} 
& \leq &
\|B_{per}\|^2_{L^\infty(\RR^d)}
\int_{\eps(Q+k)} \left| \nabla u_0^\varepsilon \right|^2
\\
& \leq &
C \eps^d + C \int_{\eps(Q+k)} \left| \eps \nabla \theta_0^\varepsilon \right|^2
\\
& \leq &
C \eps^d + C \| \eps \theta_0^\varepsilon \|^2_{H^1({\cal D})}.
\end{eqnarray*}
Using the bound~\eqref{restepsu0}, we deduce that $c_k^\eps$ converges
to 0 in $L^2({\cal D})$. In view of~\eqref{PB:phikeps_bis}, this implies
that $\phi_k^\eps$ converges to 0 in $H^1_0({\cal D})$. This concludes
the proof.
\end{proof}

\bigskip

\begin{proof}[Proof of Proposition~\ref{theo:restu1}]

As in the proof of Proposition~\ref{theo:restubar1}, the proof falls in
two steps. We first truncate $\overline{v}_k^\varepsilon$ in a function $\widetilde{v}_k^\varepsilon$ (defined by~\eqref{eq:def_vtildeeps}
below) that vanishes on $\partial \mathcal{D}$. We next estimate the difference between
$\widetilde{v}_k^\varepsilon$ and $\phi_k^\varepsilon$.

\medskip

\noindent
{\bf Step 1: Truncation of $\overline{v}_k^\varepsilon$}

Let us define $\tau_\varepsilon \in \mathcal{C}^\infty_0(\mathcal{D})$
such that $0 \leq \tau_\varepsilon(x) \leq 1$ for all $x \in
\mathcal{D}$, $\tau_\varepsilon(x) = 1$ when $\text{dist}(\partial
\mathcal{D}, x) \geq \varepsilon$ and $\varepsilon \| \nabla
\tau_\varepsilon \|_{L^\infty({\cal D})} \leq C$, where $C$ is a constant
independent of $\varepsilon$. We introduce
\begin{eqnarray*}
\mathcal{D}_\varepsilon &:=& \left\{ x \in \mathcal{D} \text{ such that }
\text{dist}(\partial \mathcal{D}, x) \geq \varepsilon \right\},
\quad
\left| {\cal D} \setminus \mathcal{D}_\varepsilon \right| \sim \eps,
\\
J_\varepsilon &:=& \left\{ k \in \ZZ^d \text{ such that }
  \varepsilon(Q+k) \cap \mathcal{D} \setminus \mathcal{D}_\varepsilon
  \ne \emptyset  \right\}, 
\quad \text{Card}(J_\varepsilon) \sim \varepsilon^{1-d},
\end{eqnarray*}
and the function $\widetilde{v}_k^\varepsilon \in H^1_0(\mathcal{D})$
defined by 
\begin{equation}
\label{eq:def_vtildeeps}
\widetilde{v}_k^\varepsilon = \varepsilon \tau_\varepsilon
\sum\limits_{i=1}^d \chi_{e_i}\left(\frac{\cdot}{\varepsilon}-k\right)
\partial_i u_0^\star,
\end{equation}
where $u_0^\star$ is solution to~\eqref{eq:homog-0} and $\chi_{e_i}$ is
solution to~\eqref{PB:chip}. Note that $\widetilde{v}_k^\varepsilon =
\overline{v}_k^\varepsilon$ except in the neighboorhood of the boundary of
$\mathcal{D}$. In the sequel, we estimate $\sum\limits_{k \in
  I_\varepsilon} \|\overline{v}_k^\varepsilon -
\widetilde{v}_k^\varepsilon\|^2_{H^1(\mathcal{D})}$, and, in the next
Step, we estimate $\sum\limits_{k \in I_\varepsilon}
\|\phi_k^\varepsilon -
\widetilde{v}_k^\varepsilon\|^2_{H^1(\mathcal{D})}$, where, we recall
(see~\eqref{def:Ieps}), 
$$
I_\varepsilon = \left\{ k \in \ZZ^d \text{ such that } \varepsilon(Q+k)
  \cap \mathcal{D} \ne \emptyset \right\}, 
\quad 
\text{Card}(I_\varepsilon) \sim \varepsilon^{-d}.
$$
Recall also that, whenever $k \notin I_\varepsilon$, we have
$\phi_k^\varepsilon \equiv 0$. 

By definition,
\begin{equation}
\nabla \overline{v}_k^\varepsilon - \nabla \widetilde{v}_k^\varepsilon 
= 
e_0^{k,\varepsilon} - e_1^{k,\varepsilon} + e_2^{k,\varepsilon}, 
\label{eq1:restu1}
\end{equation}
where
\begin{eqnarray*}
e_0^{k,\varepsilon} &=& (1-\tau_\varepsilon) \sum\limits_{i=1}^d \nabla
\chi_{e_i}\left(\frac{\cdot}{\varepsilon}-k\right) \partial_i u_0^\star, 
\\
e_1^{k,\varepsilon} &=& \varepsilon \nabla \tau_\varepsilon
\sum\limits_{i=1}^d \chi_{e_i}\left(\frac{\cdot}{\varepsilon}-k\right)
\partial_i u_0^\star, 
\\
e_2^{k,\varepsilon} &=& \varepsilon (1-\tau_\varepsilon)
\sum\limits_{i=1}^d \chi_{e_i}\left(\frac{\cdot}{\varepsilon} - k\right)
\nabla(\partial_i u_0^\star). 
\end{eqnarray*}
We now bound from above successively the $L^2$ norm of
$e_2^{k,\varepsilon}$, $e_1^{k,\varepsilon}$ and
$e_0^{k,\varepsilon}$. To this aim, the following computation will be
useful: for any $1 \leq i \leq d$, we have
$$
\sum\limits_{k \in I_\varepsilon} \int_{\mathcal{D} \setminus
  \mathcal{D}_\varepsilon}
\chi^2_{e_i}\left(\frac{\cdot}{\varepsilon}-k\right)
\leq
\sum\limits_{k \in
  I_\varepsilon} \sum\limits_{j \in J_\varepsilon} \varepsilon^d
\int_{Q+j} \chi^2_{e_i}\left(\cdot -k\right)
\leq
\sum\limits_{j \in
  J_\varepsilon} \varepsilon^d \sum\limits_{k \in I_\varepsilon}
\int_{Q+j-k} \chi^2_{e_i}.
$$
There exists $\rho$ such that 
$$
\forall \eps, \ \forall j \in J_\varepsilon, 
\ \forall k \in I_\varepsilon,
\quad 
Q+j-k \subset B(0,\rho/\eps).
$$
We thus obtain that
\begin{equation}
\sum\limits_{k \in I_\varepsilon} \int_{\mathcal{D} \setminus
  \mathcal{D}_\varepsilon}
\chi^2_{e_i}\left(\frac{\cdot}{\varepsilon}-k\right)
\leq
\sum\limits_{j \in J_\varepsilon} \varepsilon^d
\int_{B(0,\rho/\varepsilon)} \chi^2_{e_i}
\leq 
\varepsilon
\int_{B(0,\rho/\varepsilon)} \chi^2_{e_i}.
\label{eq:interm1}
\end{equation}
We next infer from~\eqref{res-b:lem4} that
\begin{equation}
\int_{B(0,\rho/\varepsilon)} \chi^2_{e_i} 
\leq 
\int_{B(0,\rho/\varepsilon)} \frac{C}{(1+|y|^{d-1})^2}dy 
\leq 
C + C \int_1^{\rho/\varepsilon} \frac{1}{r^{d-1}}dr 
\leq 
C R_{d,\eps},
\label{eq3-1:restu1}
\end{equation}
where $C$ is a constant independent of $\varepsilon$ and
\begin{equation}
\label{eq:def_Rdeps}
R_{d,\eps} :=
\left\{
\begin{array}{c}
1+\ln(1/\varepsilon) \text{ if $d=2$},
\\
1 \text{ if $d>2$}.
\end{array}
\right.
\end{equation}
Collecting~\eqref{eq:interm1} and~\eqref{eq3-1:restu1},
we deduce that
\begin{equation}
\sum\limits_{k \in I_\varepsilon} \int_{\mathcal{D} \setminus
  \mathcal{D}_\varepsilon}
\chi^2_{e_i}\left(\frac{\cdot}{\varepsilon}-k\right)
\leq 
C \eps R_{d,\eps}.
\label{eq:interm2}
\end{equation}
We now bound $e_2^{k,\varepsilon}$. 
As $u_0^\star \in
W^{2,\infty}(\mathcal{D})$, and using~\eqref{eq:interm2}, we have
\begin{eqnarray}
\sum\limits_{k \in I_\varepsilon}
\|e_2^{k,\varepsilon}\|^2_{L^2(\mathcal{D})}
&=& 
\sum\limits_{k \in I_\varepsilon} \varepsilon^2 \int_\mathcal{D}
\left[(1-\tau_\varepsilon) \sum\limits_{i=1}^d
  \chi_{e_i}\left(\frac{\cdot}{\varepsilon} - k\right) \nabla(\partial_i
  u_0^\star)\right]^2 
\nonumber \\
 &\leq& C \varepsilon^2 \|\nabla^2 u_0^\star\|^2_{L^\infty} 
\sum\limits_{k \in I_\varepsilon} 
\sum\limits_{i=1}^d \int_{\mathcal{D} \setminus {\cal D}_\eps}
\chi^2_{e_i}\left(\frac{\cdot}{\varepsilon} - k\right)
\nonumber \\
&\leq& C \eps^3 R_{d,\eps}.
\label{eq2:restu1}
\end{eqnarray}
We next turn to $e_1^{k,\varepsilon}$. The same arguments and the fact that
$\varepsilon \|\nabla \tau_\varepsilon \|_{L^\infty} \leq C$ lead to 
\begin{eqnarray}
\sum\limits_{k \in I_\varepsilon} \|e_1^{k,\varepsilon}\|^2_{L^2(\mathcal{D})} 
&\leq& 
\| \varepsilon \nabla \tau_\varepsilon \|^2_{L^\infty}
\sum\limits_{k \in I_\varepsilon} \int_{\mathcal{D} \setminus
  \mathcal{D}_\varepsilon}
\left[\sum\limits_{i=1}^d \chi_{e_i}\left(\frac{\cdot}{\varepsilon}-k\right)
  \partial_i u_0^\star \right]^2 
\nonumber \\
&\leq& 
C \sum\limits_{i=1}^d \sum\limits_{k \in
  I_\varepsilon} \int_{\mathcal{D} \setminus \mathcal{D}_\varepsilon}
\chi^2_{e_i}\left(\frac{\cdot}{\varepsilon}-k\right) 
\nonumber \\
&\leq& 
C \varepsilon R_{d,\eps},
\label{eq3:restu1}
\end{eqnarray}
where we have again used~\eqref{eq:interm2}. 
Turning to $e_0^{k,\varepsilon}$, we have, using $\nabla \chi_{e_i} \in
\left( L^2(\RR^d) \right)^d$, 
\begin{eqnarray}
\sum\limits_{k \in I_\varepsilon} \|e_0^{k,\varepsilon}\|^2_{L^2(\mathcal{D})} 
&\leq& 
C \| \nabla u_0^\star \|^2_{L^\infty} \sum\limits_{i=1}^d 
\sum\limits_{k \in I_\varepsilon}\int_{\mathcal{D} \setminus
  \mathcal{D}_\varepsilon}\left| \nabla
  \chi_{e_i}\left(\frac{\cdot}{\varepsilon}-k\right)\right|^2 
\nonumber \\
&\leq& 
C \sum\limits_{i=1}^d 
\sum\limits_{j \in J_\varepsilon} \varepsilon^d \sum\limits_{k \in I_\varepsilon}
\int_{Q+j-k} \left| \nabla \chi_{e_i} \right|^2 
\nonumber \\
&\leq&  
C \sum\limits_{i=1}^d \sum\limits_{j \in
  J_\varepsilon} \varepsilon^d \|\nabla \chi_{e_i}\|^2_{L^2(\RR^d)} 
\nonumber \\
&\leq& C \varepsilon.
\label{eq4:restu1}
\end{eqnarray}
Collecting~\eqref{eq1:restu1},~\eqref{eq2:restu1},~\eqref{eq3:restu1}
and~\eqref{eq4:restu1}, we deduce that
$$
\sum\limits_{k \in I_\varepsilon}
\|\nabla \overline{v}_k^\eps - \nabla \widetilde{v}_k^\eps
\|^2_{L^2(\mathcal{D})}
\leq 
C \left( \eps + \eps R_{d,\eps} + \eps^3 R_{d,\eps} \right),
$$
where $C$ is a constant independent of $\varepsilon$. Observing that
$$
\sum\limits_{k \in I_\varepsilon}
\|\overline{v}_k^\varepsilon - \widetilde{v}_k^\varepsilon
\|^2_{L^2(\mathcal{D})} 
\leq 
C \varepsilon^2 \|\nabla u_0^\star\|^2_{L^\infty(\mathcal{D})} 
\sum\limits_{k \in I_\varepsilon} \sum\limits_{i=1}^d 
\int_{\mathcal{D} \setminus
  \mathcal{D}_\varepsilon}
\chi^2_{e_i}\left(\frac{\cdot}{\varepsilon}-k\right)
\leq 
C \varepsilon^3 R_{d,\eps},
$$
we obtain that
\begin{equation}
\sum\limits_{k \in I_\varepsilon}
\|\overline{v}_k^\varepsilon - \widetilde{v}_k^\varepsilon
\|^2_{H^1(\mathcal{D})} 
\leq C \left( \eps + \eps R_{d,\eps} + \eps^3 R_{d,\eps} \right)
\leq
\left\{
\begin{array}{c}
C \eps \left[ 1+\ln(1/\varepsilon) \right] \text{ if $d=2$},
\\
C \eps \text{ if $d>2$},
\end{array}
\right.
\label{eq11:restu1}
\end{equation}
where $C$ is a constant independent of $\varepsilon$. 

\medskip

\noindent
{\bf Step 2:}
We next turn to estimating $\sum\limits_{k \in I_\varepsilon}
\|\phi_k^\varepsilon -
\widetilde{v}_k^\varepsilon\|^2_{H^1(\mathcal{D})}$. Using that
$A_{per}$ is coercive and
the fact that $\phi_k^\varepsilon - \widetilde{v}_k^\varepsilon \in
H^1_0(\mathcal{D})$, we have
\begin{equation}
\alpha \|\phi_k^\varepsilon -
\widetilde{v}_k^\varepsilon\|^2_{H^1(\mathcal{D})} 
\leq
\int_\mathcal{D} \left(\nabla \phi_k^\varepsilon - \nabla
  \widetilde{v}_k^\varepsilon \right)^T
A_{per}\left(\frac{\cdot}{\varepsilon}\right) \left(\nabla
  \phi_k^\varepsilon - \nabla \widetilde{v}_k^\varepsilon \right) 
=
D_0^{k,\varepsilon} + D_1^{k,\varepsilon}, 
\label{eq5:restu1}
\end{equation}
where the constant $\alpha > 0$ only depends on the coercivity constant
of $A_{per}$ and the Poincar\'e constant of the domain ${\cal D}$, and where
\begin{eqnarray*}
D_0^{k,\varepsilon} &=& 
\int_\mathcal{D} \left(\nabla \phi_k^\varepsilon - \nabla
  \widetilde{v}_k^\varepsilon \right)^T
A_{per}\left(\frac{\cdot}{\varepsilon}\right) \left(\nabla
  \phi_k^\varepsilon - \nabla \overline{v}_k^\varepsilon \right), 
\\
D_1^{k,\varepsilon} &=& 
\int_\mathcal{D} \left(\nabla \phi_k^\varepsilon - \nabla
  \widetilde{v}_k^\varepsilon \right)^T
A_{per}\left(\frac{\cdot}{\varepsilon}\right) \left(\nabla
  \overline{v}_k^\varepsilon - \nabla \widetilde{v}_k^\varepsilon
\right).
\end{eqnarray*}
We successively bound $D_0^{k,\varepsilon}$ and $D_1^{k,\varepsilon}$
from above. We begin with $D_0^{k,\varepsilon}$. Observe that, in view
of~\eqref{PB:phikeps},
\begin{multline*}
-\hbox{div}\left[A_{per}\left(\frac{\cdot}{\varepsilon}\right)\left(\nabla
    \phi_k^\varepsilon - \nabla \overline{v}_k^\varepsilon
  \right)\right] 
= 
\hbox{div}\left[\mathbf{1}_{Q+k}\left(\frac{\cdot}{\varepsilon}\right)
B_{per}\left(\frac{\cdot}{\varepsilon}\right)\left(\nabla
  u_0^\varepsilon - \sum\limits_{p=1}^d \left(e_p+ \nabla
    \wper_{e_p}\left(\frac{\cdot}{\varepsilon}\right)\right)\partial_p
  u_0^\star \right)\right] 
\\
+ \sum\limits_{p=1}^d
\hbox{div}\left[Z_k\left(\frac{\cdot}{\varepsilon}\right)\partial_p
  u_0^\star\right] + \varepsilon
\hbox{div}\left[A_{per}\left(\frac{\cdot}{\varepsilon}\right)
  \nabla(\partial_p u_0^\star)
  \chi_{e_p}\left(\frac{\cdot}{\varepsilon}-k\right) \right],
\end{multline*}
where the vector-valued function $Z_k$ is defined by
$$
Z_k(y) = \mathbf{1}_{Q+k}(y) B_{per}(y) \left(e_p + \nabla \wper_{e_p}(y)\right) +
A_{per}(y)\nabla \chi_{e_p}(y-k). 
$$
Note that, in view of~\eqref{PB:chip}, $Z_k$ is a divergence free
vector, hence 
$\dps \hbox{div}\left[Z_k\left(\frac{\cdot}{\varepsilon}\right)\partial_p
u_0^\star\right]
=
Z_k\left(\frac{\cdot}{\varepsilon}\right) \cdot \nabla \partial_p
u_0^\star$. We can thus rewrite $D_0^{k,\varepsilon}$ as
\begin{equation}
D_0^{k,\varepsilon} = D_{00}^{k,\varepsilon} + D_{01}^{k,\varepsilon} +
D_{02}^{k,\varepsilon}, 
\label{exp-B0keps}
\end{equation}
where
\begin{eqnarray*}
D_{00}^{k,\varepsilon} &=& - \int_{\mathcal{D}} \left(\nabla
  \phi_k^\varepsilon - \nabla \widetilde{v}_k^\varepsilon
\right)^T\mathbf{1}_{Q+k}\left(\frac{\cdot}{\varepsilon}\right)
B_{per}\left(\frac{\cdot}{\varepsilon}\right)\left(\nabla
  u_0^\varepsilon - \sum\limits_{p=1}^d \left(e_p+ \nabla
    \wper_{e_p}\left(\frac{\cdot}{\varepsilon}\right)\right)\partial_p u_0^\star
\right), 
\\
D_{01}^{k,\varepsilon} &=& \sum\limits_{p=1}^d \int_{\mathcal{D}}
\left(\phi_k^\varepsilon - \widetilde{v}_k^\varepsilon \right) \
Z_k\left(\frac{\cdot}{\varepsilon}\right) \cdot \nabla(\partial_p
u_0^\star),
\\
D_{02}^{k,\varepsilon} &=& - \varepsilon \sum\limits_{p=1}^d \int_{\mathcal{D}}
\left(\nabla \phi_k^\varepsilon - \nabla \widetilde{v}_k^\varepsilon
\right)^TA_{per}\left(\frac{\cdot}{\varepsilon}\right) \nabla(\partial_p
u_0^\star) \ \chi_{e_p}\left(\frac{\cdot}{\varepsilon}-k\right). 
\end{eqnarray*}
We successively bound these three quantities.
Since $\chi_{e_p} \in L^\infty(\RR^d)$ (see~\eqref{res-b:lem4}) and
$u_0^\star \in W^{1,\infty}(\mathcal{D})$, we have
$\|\widetilde{v}_k^\varepsilon\|_{L^\infty(\mathcal{D})} \leq C
\eps$. We also have that 
$\|\phi_k^\varepsilon\|_{L^\infty(\mathcal{D})} \leq C \varepsilon$, in
view of~\eqref{eq:phi_estim2} (recall indeed that $u_0^\star \in
W^{2,\infty}(\mathcal{D})$ implies that $f \in L^\infty(\mathcal{D})$,
in view of~\eqref{eq:homog-0}; assumptions of
Lemma~\ref{lem:decomposition} are thus satisfied). 
Using that $u_0^\star \in 
W^{2,\infty}(\mathcal{D})$, we now bound from above
$D_{01}^{k,\varepsilon}$: 
\begin{eqnarray*}
|D_{01}^{k,\varepsilon}| &\leq& 
\sum_{p=1}^d 
\|\phi_k^\varepsilon-\widetilde{v}_k^\varepsilon\|_{L^\infty(\mathcal{D})}
\|\nabla^2 u_0^\star\|_{L^\infty(\mathcal{D})} 
\left\|Z_k\left(\frac{\cdot}{\varepsilon}\right)\right\|_{L^1(\mathcal{D})} 
\\
&\leq& 
C \varepsilon 
\sum\limits_{p=1}^d \left[\|B_{per}\|_{L^\infty} \int_{\mathcal{D}}
  \mathbf{1}_{Q+k}\left(\frac{\cdot}{\varepsilon}\right) \left|e_p +
    \nabla \wper_{e_p}\left(\frac{\cdot}{\varepsilon}\right)\right| +
  \|A_{per}\|_{L^\infty}  \int_\mathcal{D} \left| \nabla
  \chi_{e_p}\left(\frac{\cdot}{\varepsilon}-k\right) \right| \right] 
\\
&\leq& 
C \varepsilon 
\sum\limits_{p=1}^d \left[ \varepsilon^d \|e_p + \nabla
  \wper_{e_p} \|_{L^2(Q)} 
+ \varepsilon^d 
\int_{\mathcal{D}/\varepsilon-k} |\nabla \chi_{e_p}|\right] 
\\
&\leq& 
C \varepsilon^{d+1} 
\sum\limits_{p=1}^d \left[ 1 +
\left( \int_{B(0,1)} |\nabla \chi_{e_p}| + 
\int_{B(0,\rho/\varepsilon)\setminus
  B(0,1)} |\nabla \chi_{e_p}|\right)\right].
\end{eqnarray*}
Using that $\nabla \chi_{e_p} \in (L^2(\RR^d))^d$ (see Lemma~\ref{lem4}) and the
bound~\eqref{res:lem4}, we deduce that
\begin{eqnarray*}
|D_{01}^{k,\varepsilon}|
&\leq& C \varepsilon^{d+1} 
\sum\limits_{p=1}^d \left[1 + \left( \|\nabla \chi_{e_p}\|_{L^2(\RR^d)}+
C \int_1^{1/\varepsilon} \frac{1}{r}dr\right)\right] \\
&\leq& C \varepsilon^{d+1} \left[ 1 + \ln(1/\varepsilon) \right],
\end{eqnarray*}
where $C$ is a constant independent of $\varepsilon$. We thus get
\begin{equation}
\sum\limits_{k \in I_\varepsilon} |D_{01}^{k,\varepsilon}| \leq 
C \varepsilon \left[ 1 + \ln(1/\varepsilon) \right]. 
\label{C1keps-bound}
\end{equation}
We now turn to $D_{02}^{k,\varepsilon}$. Using~\eqref{eq3-1:restu1}, we
observe that, for any $k \in I_\eps$,
$$
\left\|
\chi_{e_p}\left(\frac{\cdot}{\varepsilon}-k\right)\right\|_{L^2(\mathcal{D})}^2 
= 
\int_\mathcal{D} \left| \chi_{e_p}\left(\frac{\cdot}{\varepsilon}-k\right)
\right|^2 
=
\varepsilon^d \int_{\mathcal{D}/\varepsilon-k} |\chi_{e_p}|^2 
\leq
\varepsilon^d \int_{B(0,\bar \rho/\varepsilon)} |\chi_{e_p}|^2
\leq
C \varepsilon^d R_{d,\eps}.
$$
We thus can bound from above $D_{02}^{k,\varepsilon}$, using that $u_0^\star \in
W^{2,\infty}(\mathcal{D})$:
\begin{eqnarray}
\sum\limits_{k \in I_\varepsilon}|D_{02}^{k,\varepsilon}| 
&\leq& 
\varepsilon \|\nabla^2 u_0^\star\|_{L^\infty(\mathcal{D})} \|A_{per}\|_{L^\infty}
\sum\limits_{p=1}^d \sum\limits_{k \in I_\varepsilon}\|\nabla
\phi_k^\varepsilon - \nabla
\widetilde{v}_k^\varepsilon\|_{L^2(\mathcal{D})} \
\left\|\chi_{e_p}\left(\frac{\cdot}{\varepsilon}-k\right)\right\|_{L^2(\mathcal{D})} 
\nonumber \\
&\leq& C \varepsilon \sqrt{\sum\limits_{k \in I_\varepsilon} \|\nabla
  \phi_k^\varepsilon - \nabla
  \widetilde{v}_k^\varepsilon\|_{L^2(\mathcal{D})}^2}
\ 
\sqrt{\sum\limits_{p=1}^d \sum\limits_{k \in I_\varepsilon} \left\| 
\chi_{e_p}\left(\frac{\cdot}{\varepsilon}-k\right) \right\|_{L^2(\mathcal{D})}^2}
\nonumber \\
&\leq& C \varepsilon \sqrt{\sum\limits_{k \in I_\varepsilon} \|\nabla
  \phi_k^\varepsilon - \nabla
  \widetilde{v}_k^\varepsilon\|_{L^2(\mathcal{D})}^2} 
\sqrt{R_{d,\eps}}
\label{C2keps-bound}
\end{eqnarray}
where $C$ is a constant independent of $\varepsilon$. We next turn to
$D_{00}^{k,\varepsilon}$. Using the bound~\eqref{restepsu0} on the two
scale expansion of $u_0^\varepsilon$, we have
\begin{eqnarray}
\sum\limits_{k \in I_\varepsilon} |D_{00}^{k,\varepsilon}| &\leq&  
\|B_{per}\|_{L^\infty} \sum\limits_{k \in I_\varepsilon} \|\nabla
\phi_k^\varepsilon - \nabla
\widetilde{v}_k^\varepsilon\|_{L^2(\mathcal{D})} \left\|\nabla
  u_0^\varepsilon - \sum\limits_{p=1}^d \left(e_p+ \nabla
    \wper_{e_p} \left(\frac{\cdot}{\varepsilon}\right)\right)\partial_p
  u_0^\star\right\|_{L^2(\varepsilon(Q+k))} 
\nonumber \\
&\leq& C \sqrt{\sum\limits_{k \in I_\varepsilon}
  \|\nabla \phi_k^\varepsilon - \nabla
  \widetilde{v}_k^\varepsilon\|_{L^2(\mathcal{D})}^2}
\sqrt{\sum\limits_{k \in I_\varepsilon} \left\|\nabla u_0^\varepsilon -
    \sum\limits_{p=1}^d \left(e_p+ \nabla
      \wper_{e_p} \left(\frac{\cdot}{\varepsilon}\right)\right)\partial_p
    u_0^\star\right\|_{L^2(\varepsilon(Q+k))}^2 } 
\nonumber \\
&\leq& C \sqrt{\sum\limits_{k \in I_\varepsilon}
  \|\nabla \phi_k^\varepsilon - \nabla
  \widetilde{v}_k^\varepsilon\|_{L^2(\mathcal{D})}^2} \left\|\nabla
  u_0^\varepsilon - \sum\limits_{p=1}^d \left(e_p+ \nabla
    \wper_{e_p} \left(\frac{\cdot}{\varepsilon}\right)\right)\partial_p
  u_0^\star\right\|_{L^2(\mathcal{D})} 
\nonumber \\
&\leq& C \sqrt{\varepsilon} \sqrt{\sum\limits_{k \in I_\varepsilon}
  \|\nabla \phi_k^\varepsilon - \nabla
  \widetilde{v}_k^\varepsilon\|_{L^2(\mathcal{D})}^2}.
\label{C0keps-bound}
\end{eqnarray}
Collecting~\eqref{exp-B0keps},~\eqref{C1keps-bound},~\eqref{C2keps-bound}
and~\eqref{C0keps-bound}, we obtain that 
\begin{equation}
\sum\limits_{k \in I_\varepsilon} |D_0^{k,\varepsilon}| \leq C \left(
  \left( \sqrt{\varepsilon} + \varepsilon \sqrt{R_{d,\eps}} \right)
  \sqrt{\sum\limits_{k \in I_\varepsilon} \|\nabla \phi_k^\varepsilon -
    \nabla \widetilde{v}_k^\varepsilon\|_{L^2(\mathcal{D})}^2} +
  \varepsilon \ln(1/\varepsilon) \right). 
\label{B0keps-bound}
\end{equation}
We now turn to $D_1^{k,\varepsilon}$. Using~\eqref{eq11:restu1}, we have
\begin{eqnarray}
\sum\limits_{k \in I_\varepsilon} |D_1^{k,\varepsilon}| 
&\leq& 
C \sqrt{\sum\limits_{k \in I_\varepsilon} \|\nabla \phi_k^\varepsilon - \nabla \widetilde{v}_k^\varepsilon\|_{L^2(\mathcal{D})}^2}\sqrt{\sum\limits_{k \in I_\varepsilon} \|\nabla \overline{v}_k^\varepsilon - \nabla \widetilde{v}_k^\varepsilon\|_{L^2(\mathcal{D})}^2} \nonumber \\
&\leq& C \sqrt{\sum\limits_{k \in I_\varepsilon} \|\nabla
  \phi_k^\varepsilon - \nabla
  \widetilde{v}_k^\varepsilon\|_{L^2(\mathcal{D})}^2} \sqrt{\varepsilon
  R_{d,\eps}}. 
\label{B1keps-bound}
\end{eqnarray}
Collecting~\eqref{eq5:restu1},~\eqref{B0keps-bound} and~\eqref{B1keps-bound}, we obtain
$$
\alpha \sum\limits_{k \in I_\varepsilon}\|\phi_k^\varepsilon -
\widetilde{v}_k^\varepsilon\|^2_{H^1(\mathcal{D})} 
\leq 
C \left( \varepsilon \ln(1/\varepsilon) + 
\left( \sqrt{\varepsilon R_{d,\eps}} + \varepsilon \sqrt{R_{d,\eps}} \right)
\sqrt{\sum\limits_{k \in I_\varepsilon} \|\nabla \phi_k^\varepsilon - \nabla \widetilde{v}_k^\varepsilon\|_{L^2(\mathcal{D})}^2} \right)
$$
with, in view of~\eqref{eq:def_Rdeps}, $R_{d,\eps} = 1+ \ln(1/\eps)$ if
$d=2$, and $R_{d,\eps} = 1$ if $d>2$.
This implies that
$$
\sum\limits_{k \in I_\varepsilon}\|\phi_k^\varepsilon -
\widetilde{v}_k^\varepsilon\|^2_{H^1(\mathcal{D})} \leq C \varepsilon
\ln(1/\varepsilon),
\quad
\text{$C$ independent of $\varepsilon$.}
$$
Collecting this bound with~\eqref{eq11:restu1}, we obtain the claimed bound. 
This concludes the
proof of Proposition~\ref{theo:restu1}. 
\end{proof}

\bigskip

\noindent
\textbf{Acknowledgements:} Support from EOARD under Grant
FA8655-10-C-4002 is gratefully acknowledged. We also thank Xavier Blanc
for useful discussions.

\end{document}